\documentclass[11pt]{amsart}
\usepackage[utf8]{inputenc} \usepackage[T1]{fontenc}
\usepackage{lmodern}
\usepackage[margin=2.5cm, centering]{geometry}

\usepackage[english]{babel}
\usepackage{geometry}
\usepackage{caption}
\usepackage{subcaption}
\usepackage{mathrsfs,amsmath,amsfonts,amstext,amssymb,amsthm,amsopn,amsthm,color, dsfont,bm,bbm}
\usepackage{enumitem}
\usepackage[titletoc]{appendix}
\usepackage[dvipsnames]{xcolor}
\usepackage{stackrel}
\usepackage{bigints}

\usepackage{todonotes}

\usepackage{thmtools}
\usepackage[colorlinks,citecolor=blue,urlcolor=blue,bookmarks=true]{hyperref}
\hypersetup{
	pdfpagemode=UseNone,
	pdfstartview=FitH,
	pdfdisplaydoctitle=true,
	pdfborder={0 0 0}, 
	pdftitle={Transition densities of spectrally positive L\'{e}vy processes},
	pdfauthor={\L{}ukasz Le\.{z}aj},
	pdflang=en-GB
}

\numberwithin{equation}{section}
\declaretheorem[name=Theorem, numberwithin=section]{theorem}
\newtheorem{lemma}[theorem]{Lemma}
\newtheorem{proposition}[theorem]{Proposition}
\newtheorem{corollary}[theorem]{Corollary}

\newtheorem{claim}[theorem]{Claim}

\declaretheoremstyle[bodyfont=\normalfont]{remark-style}
\declaretheorem[name={Remark}, style=remark-style, sibling=theorem]{remark}
\declaretheorem[name={Example}, style=remark-style, sibling=theorem]{example}

\newcommand{\RR}{\mathbb{R}}
\newcommand{\NN}{\mathbb{N}} 
\newcommand{\EE}{\mathbb{E}}
\newcommand{\PP}{\mathbb{P}}

\newcommand{\CC}{\mathbb{C}}

\newcommand{\calB}{\mathcal{B}}

\newcommand{\calL}{\mathcal{L}}

\newcommand{\bfT}{\mathbf{T}}
\newcommand{\bfX}{\mathbf{X}}

\newcommand{\vphi}{\varphi}

\renewcommand{\epsilon}{\varepsilon}
\renewcommand{\leq}{\leqslant}
\renewcommand{\geq}{\geqslant}
\renewcommand{\Re}{\operatorname{Re}}

\newcommand{\ind}[1]{{\mathds{1}_{{#1}}}}

\newcommand{\WUSC}[3]{\textrm{\rm WUSC}\left(#1,#2,#3\right)}
\newcommand{\WLSC}[3]{\textrm{\rm WLSC}\left(#1,#2,#3\right)}

\begin{document}
\title{Transition densities of spectrally positive L\'{e}vy processes}
\author{\L{}ukasz Le\.{z}aj}
\address{
		\L{}ukasz Le\.{z}aj\\
        Wydzia\l{} Matematyki,
        Politechnika Wroc\l{}awska\\
        Wyb. Wyspia\'{n}skiego 27\\
        50-370 Wroc\l{}aw\\
        Poland}
\email{lukasz.lezaj@pwr.edu.pl}

\thanks{The author was partially supported by the National Science Centre (Poland): grant 2016/23/B/ST1/01665.}

\begin{abstract}
	We prove asymptotic behaviour of transition density for a large class of spectrally one-sided L\'{e}vy processes of unbounded variation satisfying mild condition imposed on the second derivative of the Laplace exponent, or equivalently, on the real part of the characteristic exponent. We also provide sharp two-sided estimates on the density when restricted additionally to processes without Gaussian component.
\end{abstract}
\maketitle

\section{Introduction}\label{sec:1}
The aim of this article is to discuss the behaviour of the transition density of spectrally one-sided L\'{e}vy processes of unbounded variation, or, in other words, spectrally one-sided processes which are not subordinators with drift. Such processes, due to their specific structure, find natural applications in financial models, in particular insurance risk modelling and queue theory and therefore they have been intensively analysed from that point of view. The prominent example here and, at the same time, one of the first questions one would like to ask in the financial setting is the so-called exit problem --- the identification of the distribution of the pair $(\tau_I, X_{\tau_I})$, where $I$ is an open interval --- which has been intensively discussed over last decades. One should list here prominent works of Zolotarev \cite{Zolotarev64}, Tak\'{a}cs \cite{Takacs67}, Emery \cite{Emery73} and Rogers \cite{Rogers90}. Vast majority of results is expressed in terms of so-called scale functions, which have been of independent interest later on, see e.g. \cite{HK2011} or \cite{KKER12}. Also, specific structure of spectrally one-sided processes considerably simplifies the fluctuation theory, which for general L\'{e}vy processes is rather implicit. For details we refer to books of Bertoin \cite[Sections VI and VII]{Bertoin96}, Kyprianou \cite{Kyprianou2006} or Sato \cite{Sato}. That short list is far from being complete and for further discussion we refer to the works above and the references therein.

The abundant number of articles related to financial applications stays in stark contrast with the fact that surprisingly little is known about transition density of general spectrally one-sided L\'{e}vy processes, although it seems that such knowledge could potentially be an important and useful tool. One may find e.g. asymptotic series expansion for the special case of stable processes in the book of Zolotarev \cite[Theorem 2.5.2]{Zolotarev86}, but, to the author's best knowledge, general results has not been obtained. Therefore, the purpose of this article is to fill that gap in the theory and analyse the behaviour of the transition densities of spectrally one-sided L\'{e}vy processes in a feasibly wide generality.

Let us briefly describe our results. From theoretical point of view, absence of negative (positive) jumps allows us to exploit techniques involving the Laplace transform, which can be easily proved to exist (see e.g. the book of Bertoin \cite[Section VII]{Bertoin96}). We exploit that property in the first part, where we concentrate on derivation of asymptotic behaviour of the transition density, which is covered by Theorem \ref{thm:1}. Note that the result is very general, as the only assumption is the lower scaling property with the index $\alpha>0$. In particular, we assume neither upper scaling property nor absolute continuity of the L\'{e}vy measure $\nu(dx)$. Observe that Brownian motion with an independent subordinator satisfying scaling condition is admissible. The case $\alpha=2$ is also included. Recall that without Gaussian component, the condition of having unbounded variation is tantamount to satisfying the integral condition (see Preliminaries for elaboration)
\begin{equation}\label{eq:61}
\int_{(0,\infty)}(1\wedge x)\,\nu(dx)=\infty.
\end{equation}
In fact, if this is the case, the assumption $\alpha>0$ may  seem superfluous at first sight, as the integrability condition \eqref{eq:61}, roughly speaking, requires enough singularity of order at least 1. One may, however, construct a bit pathological example of a L\'{e}vy process with unbounded variation but with lower scaling index strictly smaller than 1. The reader is referred for details to Remark \ref{rem:1}, but we highlight here that such processes are also included. Let us note in passing that, in some cases, it is easier to impose scaling condition on the real part of the characteristic exponent instead of second derivative of the Laplace exponent. These two are in fact equivalent and we state that result in Corollary \ref{cor:5}. If any of them is true, then by \eqref{eq:44} and \eqref{eq:48}, we have, for some $x_0 \geq 0$,
\[
x^2\vphi''(x) \approx \Re \psi(x), \quad x \geq x_0.
\]
Here $\vphi$ is the Laplace exponent and $\psi$ -- the characteristic exponent of of L\'{e}vy process. Admittedly, one of strengths of Theorem \ref{thm:1} lies in the fact that the expression in the exponent is given explicitly and there is no hidden, unknown constant. Nonetheless, in view of the equation above, one can substitute the Laplace exponent with the characteristic exponent, if necessary, at the cost of losing exact formula and implicit constant which will appear instead.

Next, we restrict ourselves to processes of unbounded variation without Gaussian component and focus on upper and lower estimates on the transition density. While the former, covered by Theorem \ref{thm:2}, does not require additional assumptions and is independent of previous results, the latter, consisting of Lemma \ref{lem:3} and Lemma \ref{lem:4} require apparently stronger conditions, i.e. $\alpha \geq 1$ and $\alpha>1$, respectively. They provide local and tail lower estimates on the transition density, and the proof of the latter relies strongly on Theorem \ref{thm:1}. As above, we point out that the condition $\alpha \geq 1$ is not very restrictive in the class of processes with unbounded variation. Finally, we merge all previous results in Theorem \ref{thm:4} in order to obtain sharp two-sided estimates. They require both lower and upper scaling condition with indices strictly separated from 1 and 2, i.e. $1<\alpha \leq \beta <2$. Let us note here that in contrast to symmetric processes, where a lot is already known, a general non-symmetric case is still under development. It usually requires either implying familiar structure or imposing complex assumptions on the process. See e.g. \cite{GS2019, KKPS15, KKPS17, VK13, VKAK13, Picard97}, and the references therein. Recently, estimates for subordinators in a general setting were obtained (see e.g. \cite{CPKW20, ChoKim19, GLT2018}). Still, we are not aware of any article which would treat transition densities of general spectrally one-sided L\'{e}vy process of unbounded variation in a comprehensive way. By inspecting the proofs of Lemmas \ref{lem:3} and \ref{lem:4}, Theorem \ref{thm:3} and Proposition \ref{prop:11} we see that covering the limit cases using our methods is not possible, and it is not very surprising, as in general, they usually require more sophisticated methods, or sometimes even a completely different approach. Nonetheless, the asymptotic behaviour displayed by Theorem \ref{thm:1} covers both $\alpha=1$ and $\alpha=2$.

The article is organized as follows. In Section \ref{sec:prelims} we introduce our setting and prove some basic properties of the Laplace exponent $\vphi$. Section \ref{sec:asymptotics} is devoted to the proof of the asymptotic behaviour of the transition density. Upper and lower estimates are derived in Section \ref{sec:estimates}, while in Section \ref{sec:sharp} we combine all previous results to obtain sharp two-sided estimates on the transition density.

\subsection*{Acknowledgements}
I would like to express my deep gratitude to Tomasz Grzywny for spending a long time of helpful discussions, numerous valuable remarks and for reading the first version of the manuscript.
\section{Preliminaries}\label{sec:prelims}
\subsection*{Notation.} Throughout the paper $c, c_1,C_1,...$ denote positive constants which may vary from line to line. By $c=c(a)$ we mean that the constant $c$ depends only on one parameter $a$. For two functions $f,g\colon (0,\infty) \mapsto [0,\infty]$ we write $f \approx g$ if the quotient $f/g$ stays between two positive constants. Analogous rule is applied to symbols $\lesssim, \gtrsim$. We set $a \wedge b=\min \{a,b\}$ and $a\vee b=\max \{a,b\}$. Borel sets on $\RR$ are denoted by $\calB(\RR)$.

Let $\bfX=(X_t \colon t \geq 0)$ be a spectrally positive L\'{e}vy process, that is a L\'{e}vy process on $\RR$ with only positive jumps, which is not a subordinator with drift. There is a function $\psi \colon \RR \mapsto \CC$ such that for all $t > 0$ and $\xi \in \RR$,
\[
\EE e^{i\xi X_t} = e^{-t\psi(\xi)}.
\]
There are $\sigma \geq 0$, $b \in \RR$ and $\sigma$-finite measure $\nu$ on $(0,\infty)$ satisfying
\[
\int_{(0,\infty)} \big( 1 \wedge x^2 \big) \,\nu(dx)<\infty,
\]
such that for all $\xi \in \RR$,
\begin{equation}\label{eq:14}
\psi(\xi) = \sigma^2\xi^2 -i\xi b - \int_{(0,\infty)}\big(e^{i\xi x}-1 - i\xi x\ind{x<1}\big)\,\nu(dx).
\end{equation}
By $\vphi$ we denote the Laplace exponent of $\bfX$, i.e.
\begin{equation}\label{eq:15}
\EE e^{-\lambda X_t} = e^{t\vphi(\lambda)}, \quad \lambda \geq 0.
\end{equation}
By holomorphic extension one can see that
\begin{align}\label{eq:16}
\vphi(\lambda) = \sigma^2 \lambda^2-b\lambda+ \int_{(0,\infty)} \big( e^{-\lambda x} - 1 + \lambda x \ind{x<1} \big)\,\nu(dx).
\end{align}
Let us introduce a symmetric, continuous and non-decreasing majorant of $\Re \psi$, i.e.
\[
\psi^*(r) = \sup_{|z|\leq r} \Re \psi(z), \quad r>0,
\]
and its generalized inverse function
\[
\psi^{-1}(s) = \sup \{ r>0\colon \psi^*(r)=s \}.
\]
Then we have
\[
\psi^* \big( \psi^{-1}(s) \big) = s, \qquad \psi^{-1} \big( \psi^*(s) \big) \geq s.
\]
Observe that if
\[
\int_{(0,1)}x\,\nu(dx)<\infty,
\]
then the underlying process has almost surely bounded variation and we can rewrite \eqref{eq:14} in the following way:
\[
\psi(\xi) =\sigma^2\xi^2 -i\xi b_0 - \int_{(0,\infty)} \big( e^{i\xi x}-1 \big) \,\nu(dx),
\]
where
\[
b_0 = b+\int_{(0,1)}x\,\nu(dx).
\]
If this is the case, then $X_t = B_t + b_0t+T_t$, where $\mathbf{B}$ is a Brownian motion and $\bfT = (T_t \colon t \geq 0)$ is a subordinator, that is a one-dimensional L\'{e}vy process with non-decreasing paths which starts from 0. Although it is usually clear from the setting, in order to avoid unnecessary confusions, we assume {\bf in the whole paper} that $\bfX$ is of unbounded variation. By \cite[Theorems 21.9 and 24.10]{Sato}, this is true either when $\sigma>0$ or the following condition is satisfied:
\begin{equation}\label{eq:34}
\int_{(0,1)}x\,\nu(dx)=\infty.
\end{equation}
Therefore, in this article we exclude the case, when $\bfX$ is a subordinator with drift. Brownian motion with an independent subordinator, however, is included. For results concerning heat kernel estimates for subordinators we refer the reader to \cite{GLT2018} and the references therein.

By differentiating \eqref{eq:16} twice one can easily deduce that $\vphi$ is convex and since, in view of \cite[Theorem 24.10]{Sato}, $\PP(X_1\leq-1)>0$, we have $\vphi(\lambda) \to \infty$ as $\lambda \to \infty$. It is not, however, necessarily positive. Indeed, by differentiating \eqref{eq:15} with respect to $\lambda$, setting $t=1$ and taking the limit $\lambda \to 0^+$ yields
\begin{equation}\label{eq:30}
\EE X_1 = -\vphi'(0^+) = b+\int_{[1,\infty)} x \,\nu(dx).
\end{equation}
We see that $\EE X_1 \in (-\infty,+\infty]$. In particular, if $\EE X_1>0$ then $\vphi < 0$ in a neighbourhood of the origin. Let $\theta_0$ be the largest root of $\vphi$. Similarly, let
\[
\theta_1 = \inf \big\lbrace s>0\colon \vphi'(s)>0 \big\rbrace
\]
We have $\theta_1 \leq \theta_0$ and the equality may occur only for the case $\theta_0=\theta_1=0$. Note that there is always a root of $\vphi$ at $\lambda =0$ and, due to convexity of $\vphi$, at most one root for $\lambda>0$, precisely in the case $\theta_1>0$. By the Wiener-Hopf factorization we get that $\vphi$ is necessarily of the form
\begin{align}\label{eq:31}
\vphi(\lambda) = (\lambda-\theta_0)\phi(\lambda),
\end{align}
where $\phi$ is a Laplace exponent of a (possibly killed) subordinator, known as an ascending ladder height process, with the L\'{e}vy measure $\gamma$ given by
\[
\gamma((x,\infty)) = e^{\theta_0}x \int_x^{\infty} e^{-\theta_0 u} \nu((u,\infty)) \,d u.
\]
See \cite[Section 4]{HK2011} and the references therein. It is known that $\phi$ is a Bernstein function, i.e. a function in $C^{\infty}$ such that its derivative is completely monotone (see e.g. \cite{SSVBernstein} for a thorough analysis).

Following Pruitt \cite{Pruitt81}, we define concentration functions $K$ and $h$ by setting
\[
K(r) = \frac{\sigma^2}{r^2}+\frac{1}{r^2} \int_{(0,r)} s^2 \, \nu(ds), \quad r>0,
\]
and
\[
h(r) = \frac{\sigma^2}{r^2}+\int_{(0,\infty)}  \bigg( 1 \wedge \frac{s^2}{r^2} \bigg) \,\nu(ds), \quad r>0.
\]
Clearly, $h(r) \geq K(r)$. Notice that by the Fubini-Tonelli theorem,
\begin{equation}\label{eq:42}
h(r) = 2\int_r^{\infty} K(s)s^{-1} \,d s.
\end{equation}
Furthermore, by \cite[Lemma 4]{Grzywny14}, for all $r>0$,
\begin{equation}\label{eq:43}
\frac{1}{24}h \big( r^{-1} \big) \leq \psi^*(r) \leq 2h \big( r^{-1} \big).  
\end{equation}
Moreover, if the L\'{e}vy measure $\nu(dx)$ has a monotone density $\nu(x)$, then obviously, for any $r>0$,
\begin{equation}\label{eq:78}
r\nu(r) \leq K(r) \leq h(r).
\end{equation}

Let us introduce notions of almost monotonicity and scaling property, which are ubiquitous in the paper. First, a function $f\colon [x_0,\infty) \mapsto [0,\infty)$ is \emph{almost increasing} on $(x_1,\infty)$ for some $x_1 \geq x_0$ if there is $c \in (0,1]$ such that for all $y\geq x > x_1$,
\[
f(y) \geq cf(x).
\]
Similarly, it is \emph{almost decreasing} on $(x_1,\infty)$ if there is $C \geq 1$ such that for all $y \geq x >x_1$,
\[
f(y) \leq Cf(x).
\]

We say that a function $f \colon [x_0,\infty) \mapsto [0,\infty)$ has \emph{weak lower scaling property} at infinity, if there are $\alpha \in \RR$, $c \in (0,1]$ and $x_1 \geq x_0$ such that for all $\lambda \geq 1$ and $x>x_1$,
\[
f(\lambda x) \geq c\lambda^{\alpha}f(x).
\]
In such case we write $f \in \WLSC{\alpha}{c}{x_1}$. Analogously, we say that $f$ has a \emph{weak upper scaling property} at infinity ($f \in \WUSC{\beta}{C}{x_1}$), if there are $\beta \in \RR$, $C \geq 1$ and $x_1 \geq x_0$ such that for all $\lambda \geq 1$ and $x>x_1$,
\[
f(\lambda x) \leq C\lambda^{\beta}f(x).
\]
A convenient equivalent definition of scaling property is provided \cite{BGR14}. Namely, $f \in \WLSC{\alpha}{c}{x_1}$ if and only if the function
\[
(x_1,\infty) \ni x \mapsto f(x)x^{-\alpha}
\]
is almost increasing. Analogously, $f \in \WUSC{\beta}{C}{x_1}$ if and only if the function
\[
(x_1,\infty) \ni x \mapsto f(x)x^{-\beta}
\]
is almost decreasing. Finally, a function $f \colon [x_0,\infty) \mapsto [0,\infty)$ has \emph{a doubling property} on $(x_1,\infty)$ if there are $x_1 \geq x_0$ and $C \geq 1$ such that for all $x > x_1$,
\[
C^{-1}f(x) \leq f(2x) \leq Cf(x).
\]
Note here that a non-increasing function with weak lower scaling property and non-decreasing function with weak upper scaling property have a doubling property. In particular, $\vphi'' \in \WLSC{\alpha-2}{c}{x_0}$ has a doubling property.
\subsection{Properties of the Laplace exponent \texorpdfstring{$\vphi$}{}}
First let us observe that by differentiating \eqref{eq:16} and using the fact that for all $x>0$,
\[
xe^{-x} \leq 1-e^{-x}, 
\]
  we get that for all $\lambda \geq 0$,
\begin{equation}\label{eq:20}
\vphi(\lambda) \leq \lambda \vphi'(\lambda).
\end{equation}
Furthermore, if $\theta_0>0$ then $-\vphi$ is positive and concave on $(0,\theta_1)$. Thus, for all $x \leq \theta_1$ and $\lambda \leq 1$,
\[
-\vphi(x)- \big( -\vphi(\lambda x) \big) \leq (1-\lambda)x \big( -\vphi'(\lambda x) \big).
\]
Thus, by \eqref{eq:20}, for all $x \leq \theta_1$ and $\lambda \leq 1$,
\begin{equation}\label{eq:47}
\lambda \big( -\vphi(x)\big) \leq  -\vphi(\lambda x).
\end{equation}
\begin{proposition}\label{prop:1}
	There are $C_1, C_2 \geq 1$ such that $\vphi' \in \WUSC{1}{C_1}{2\theta_1}$ and $\vphi \in \WUSC{2}{C_2}{2\theta_0}$. Furthermore, if $\theta_0 = 0$ and $\vphi'(0)\leq 0$ then $C_1=C_2=1$, i.e. for all $x>0$ and $\lambda \geq 1$,
	\begin{equation}\label{eq:46}
	\vphi'(\lambda x) \leq \lambda \vphi'(x) \qquad \text{and} \qquad \vphi(\lambda x) \leq \lambda^2\vphi(x).
	\end{equation}
\end{proposition}
\begin{proof}
	Let $\lambda \geq 1$. First observe that by monotonicity of $\vphi''$, for all $x>\theta_1$,
	\begin{align*}
	\vphi'(\lambda x)-\vphi'(\lambda \theta_1) = \int_{\lambda \theta_1}^{\lambda x} \vphi''(s) \, ds &= \lambda \int_{\theta_1}^{x} \vphi''(\lambda s)\, ds \\ &\leq \lambda \int_{\theta_1}^{x}\vphi''(s)\, ds  = \lambda \vphi'(x).
	\end{align*}
	Thus, we get the claim for $\vphi'$ in the case $\theta_1=0$ and $\vphi'(0)\leq 0$. If $\theta_1=0$ but $\vphi'(0)>0$ then we clearly have $\vphi'(0) \leq \lambda \vphi'(x)$ for all $x >0$ and $\lambda \geq 1$, and we get the claim with $C_1=2$. Finally, if $\theta_1>0$ then it remains to prove that there is $c > 0$ such that for all $x>2\theta_1$ and $\lambda \geq 1$, 
	\begin{align}\label{eq:19}
	\vphi'(\lambda \theta_1) \leq c\lambda \vphi'(x).
	\end{align}
	Since $\vphi'(\theta_1)=0$, by monotonicity of $\vphi''$ and $\vphi'$ we obtain
	\[
	\vphi'(\lambda \theta_1) = \int_{\theta_1}^{\lambda \theta_1} \vphi''(s) \, d s \leq (\lambda-1)\theta_1 \vphi''(\theta_1) \leq \lambda \theta_1\vphi''(\theta_1) \leq c\lambda\vphi'(x),
	\]
	where $c=(\theta_1\vphi''(\theta_1))/\vphi'(2\theta_1)$, and \eqref{eq:19} follows. Now, with the first part proved, a similar argument applies to the second and therefore it is omitted.
\end{proof}
\begin{proposition}\label{prop:5}
		There is $C = C(\vphi)\geq 1$ such that for all $x>2\theta_0$ we have
		\begin{equation}\label{eq:21}
		\vphi(x) \leq x\vphi'(x) \leq C\vphi(x).
		\end{equation}
		Furthermore, for all $x>2\theta_1$,
		\[
		2\vphi'(x) \geq x\vphi''(x).
		\]
\end{proposition}
\begin{proof}
	We first observe that the first inequality of \eqref{eq:21} follows from \eqref{eq:20}.	Let $x>2\theta_0$ and $1\leq b<a$. By monotonicity of $\vphi'$,
	\[
	\vphi(ax)-\vphi(bx) \geq x(a-b)\vphi'(bx).
	\]
	Put $b=1$ and $a=2$. By Proposition \ref{prop:1},
	\[
	\frac{x\vphi'(x)}{\vphi(x)} \leq \frac{\vphi(2x)}{\vphi(x)}-1 \leq 4\tilde{C}-1,
	\]
	where $\tilde{C}$ is taken from Proposition \ref{prop:1}, and the first part follows. For the proof of the second part it remains to observe that by monotonicity of $\vphi''$, for $x>2\theta_1$,
		\[
	\frac{\vphi'(x)}{x\vphi''(x)} \geq \frac{\vphi'(x)-\vphi'(\theta_1)}{x\vphi''(x)}= \frac1x \int_{\theta_1}^x \frac{\vphi''(s)}{\vphi''(x)} \, d s \geq 1-\frac{\theta_1}{x}.
	\] 
	Thus, for $x>2\theta_1$ we get the claim.
\end{proof}
\begin{corollary}\label{cor:3}
	There is $c = c(\vphi) > 0$ such that for all $x \in (0,\theta_0/2) \cup (2\theta_0,\infty)$,
	\begin{align*}
	\big\lvert \vphi(x) \big\rvert \geq c x^2 \vphi''(x).
	\end{align*}
	The implied constant $c$ depends only on $\theta_0$.
\end{corollary}
\begin{proof}
	In view of Proposition \ref{prop:5}, it remains to prove that if $\theta_0>0$ then there is $c>0$ such that for all $x<\theta_0/2$,
	\begin{equation}\label{eq:33}
	-\vphi(x) \geq cx^2\vphi''(x).
	\end{equation}
	From \eqref{eq:31}, we have
	\[
	\vphi''(x) = 2\phi'(x) + (\theta_0-x)\big(-\phi''(x)\big).
	\]
	By \cite[Lemma 3.9.34]{Jacob}, for any $n \in \NN_+$,
	\begin{equation*}
	\phi(\lambda) \geq \frac{(-1)^{n+1}}{n!}\lambda^n \phi^{(n)}(\lambda), \quad \lambda > 0.
	\end{equation*}
	Hence, for $x<\theta_0/2$,
	\begin{equation}\label{eq:32}
	-\vphi(x) = (\theta_0-x) \phi(x) \gtrsim \phi(x) \geq x\phi'(x) \gtrsim x^2\phi'(x). 
	\end{equation}
	Moreover,
	\[
	(\theta_0-x) \phi(x) \gtrsim (\theta_0-x) x^2\big(-\phi''(x)\big),
	\]
	which together with \eqref{eq:32} imply \eqref{eq:33}, and the claim follows.
\end{proof}

Now we deduce some properties of $\vphi$ and its derivatives which follow from scaling properties.
\begin{proposition}\label{prop:2}
	Suppose $\vphi'' \in \WLSC{\alpha-2}{c}{x_0}$ for some $c \in (0,1]$, $x_0 \geq 0$ and $\alpha > 0$. Then $\vphi' \in \WLSC{\alpha-1}{c}{x_0 \vee \theta_1}$ and $\vphi \in \WLSC{\alpha}{c}{x_0 \vee \theta_0}$.
\end{proposition}
\begin{proof}
	We proceed as in the proof of Proposition \ref{prop:1}. Let $\lambda \geq 1$ and $x>\theta_1 \vee x_0$. By the weak scaling property of $\vphi''$,
	\begin{align*}
	\vphi'(\lambda x) &\geq \vphi'(\lambda x)-\vphi'(\lambda (\theta_1 \vee x_0)) = \int_{\lambda (\theta_1 \vee x_0)}^{\lambda x} \vphi''(s)\, d s = \lambda \int_{\theta_1 \vee x_0}^x \vphi''(\lambda s)\, d s \\ &\geq c\lambda^{\alpha-1} \int_{\theta_1 \vee x_0}^x \vphi''(s) \, d s = c\lambda^{\alpha-1} \big( \vphi'(x)-\vphi'(\theta_1 \vee x_0) \big),
	\end{align*}
	and the claim follows for the case $\theta_1 \geq x_0$ and $\vphi'(\theta_1) \leq 0$. If $\theta_1<x_0$ then by differentiating \eqref{eq:16} we conclude that $\vphi'(\lambda) \to \infty$ as $\lambda \to \infty$. Since it is also monotone, there is $x_1>x_0$ such that $\vphi'(x) \geq 2\vphi'(x_0)$ for all $x>x_1$, and consequently, $\vphi' \in \WLSC{\alpha-1}{\tilde{c}}{x_1}$ for some $\tilde{c} \in (0,1]$. Finally, using continuity and positivity of $\vphi'$ we can extend the scaling area to $(x_0,\infty)$ at the expense of worsening the constant. The same reasoning applies for the case $\theta_1=0$ and $\vphi'(0) > 0$. The proof of the weak scaling property of $\vphi$ follows by an analogous argument.
\end{proof}

\begin{proposition}\label{prop:3}
	Let $\vphi$ be a Laplace exponent of a spectrally positive L\'{e}vy process of infinite variation such that $\vphi'(\theta_1)=0$. Then $\vphi' \in \WLSC{\tau}{c}{x_0 \vee \theta_1}$ for some $c \in (0,1]$, $x_0 \geq 0$ and $\tau >0$ if and only if $\vphi'' \in \WLSC{\tau-1}{c'}{x_0}$ for some $c' \in (0,1]$. Furthermore, if $\vphi' \in \WLSC{\tau}{c}{x_0 \vee \theta_1}$ then there is $C \geq 1$ such that for all $x>x_0 \vee 2\theta_1$,
	\begin{equation}\label{eq:63}
	C^{-1}\vphi'(x) \leq x\vphi''(x) \leq C\vphi'(x).
	\end{equation}
\end{proposition}

\begin{proof}
	Assume first that $\vphi'' \in \WLSC{\tau-1}{c'}{x_0}$. We claim that \eqref{eq:63} holds true. In view of Proposition \ref{prop:5}, it is enough to prove the first inequality. First, let $x>2\theta_1 \vee x_0$. By the weak scaling property of $\vphi''$,
	\begin{align*}
	\frac{\vphi'(x)-\vphi'(\theta_1 \vee x_0)}{x\vphi''(x)} = \frac1x \int_{\theta_1 \vee x_0}^x \frac{\vphi''(s)}{\vphi''(x)}\, d s \leq \frac{1}{c'x^{\tau}}\int_{\theta_1 \vee x_0}^x s^{\tau-1} \, d s \leq \frac{1}{c'\tau}.
	\end{align*}
	Thus, we get \eqref{eq:63} if $x_0 \leq \theta_1$. If this is not the case, observe that since $\vphi'' \in \WLSC{\tau-1}{c}{x_0}$ and $\tau > 0$, the function
	\[
	(x_0,\infty) \ni x \mapsto x\vphi''(x)
	\]
	is almost increasing. Thus, for $x>2x_0>0$,
	\[
	x\vphi''(x) \geq c 2x_0 \vphi''(2x_0).
	\]
	Since $\vphi''$ is continuous and positive on $[x_0,2x_0]$, we get that $x\vphi''(x) \gtrsim 1$ for all $x>x_0$, and \eqref{eq:63} follows. The scaling property of $\vphi'$ is now an immediate consequence.
	
	Now assume $\vphi' \in \WLSC{\tau}{c}{x_0}$. By monotonicity of $\vphi''$, for $0<b<a$,
	\[
	\frac{\vphi'(ax)-\vphi'(bx)}{\vphi'(x)} \leq \frac{x(a-b)\vphi''(bx)}{\vphi'(x)}.
	\]
	Put $b=1$. Then by the scaling property of $\vphi'$,
	\begin{align*}
	\frac{x(a-1)\vphi''(x)}{\vphi'(x)} \geq \frac{\vphi'(ax)}{\vphi'(x)}-1 \geq ca^{\tau}-1,
	\end{align*}
	for all $x>x_0$. Thus, for $a=2^{1/\tau}c^{-1/\tau}$ we obtain that $\vphi'(x) \lesssim x\vphi''(x)$ for all $x>x_0$, which, combined with Proposition \ref{prop:5}, yields \eqref{eq:63}, and the scaling property of $\vphi''$ follows. That completes the proof. 
\end{proof}

Combining Propositions \ref{prop:5},  and \ref{prop:3}, we immediately obtain the following corollary.
\begin{corollary}\label{cor:2}
	Let $\vphi$ be a Laplace exponent of a spectrally positive L\'{e}vy process of infinite variation such that $\vphi'(\theta_1)=0$. Then $\vphi \in \WLSC{\alpha}{c}{x_0 \vee \theta_0}$ for some $c \in (0,1]$, $x_0 \geq 0$ and $\alpha>1$ if and only if $\vphi'' \in \WLSC{\alpha-2}{c'}{x_0}$ for some $c' \in (0,1]$. Furthermore, if $\vphi \in \WLSC{\alpha}{c}{x_0 \vee \theta_0}$ then there is $C \geq 1$ such that for all $x>x_0 \vee 2\theta_0$
	\[
	C'^{-1}\vphi(x) \leq x^2\vphi''(x) \leq C'\vphi(x).
	\]
\end{corollary}
\begin{lemma}\label{lem:2}
	Suppose $\vphi'' \in \WLSC{\alpha-2}{c}{x_0}$ for some $c \in (0,1]$, $x_0 \geq 0$ and $\alpha>0$. There is a constant $C>0$ such that for all $x>x_0$,
	\[
	C\vphi''(x) \leq \sigma^2+ \int_{(0,1/x)}s^2 \,\nu(ds).
	\]
\end{lemma}
\begin{proof}
	First assume $\sigma=0$; the extension to any $\sigma$ is immediate. Let $f \colon (0,\infty) \mapsto \RR$ be a function defined as
	\[
	f(t) = \int_{(0,t)} s^2\, \nu(ds).
	\]
	Observe that by the Fubini-Tonelli theorem, for $x>0$ we have
	\begin{align*}
	\calL f(x) &= \int_0^{\infty}e^{-xt} \int_{(0,t)}s^2\,\nu(ds) \, d t = \int_{(0,\infty)}s^2 \int_s^{\infty} e^{-xt} \, d t \,\nu(ds) = x^{-1}\vphi''(x).
	\end{align*}
	Since $f$ is non-decreasing, for any $s>0$,
	\begin{align*}
	\vphi''(x) =x\calL f(x) \geq \int_s^{\infty} e^{-t}f(t/x) \, d t \geq e^{-s} f(s/x).
	\end{align*}
	Hence, for any $u >2$,
	\begin{align*}
	\vphi''(x) &= \int_0^u e^{-s}f(x/s)\, d s + \int_u^{\infty} e^{-s}f(x/s)\, d s \leq f(u/x) + \int_u^{\infty} e^{-s/2} \vphi''(x/2) \, d s.
	\end{align*}
	Therefore, setting $x=\lambda u >2x_0$, by the weak scaling property of $\vphi''$,
	\begin{align*}
	f(1/\lambda) &\geq \vphi''(u\lambda) - 2e^{-u/2}\vphi''(u\lambda/2) \geq \big( 2^{\alpha-2}c-2e^{-u/2} \big)\vphi''(u\lambda/2).
	\end{align*}
	At this stage, we select $u>2$ such that
	\[
	2^{\alpha-2}c-2e^{-u/2} \geq 2^{-2}c.
	\]
	Then again, by the weak scaling property of $\vphi''$, for $\lambda>x_0$,
	\[
	f(1/\lambda) \geq c2^{-2}\vphi''(u\lambda/2) \geq (c/u)^2\vphi''(\lambda),
	\]
	which ends the proof.
\end{proof}
Since
\[
K(1/x) \leq ex^2\vphi''(x),
\]
by Lemma \ref{lem:2} we immediately obtain the following corollary.
\begin{corollary}\label{cor:4}
	Suppose $\vphi'' \in \WLSC{\alpha-2}{c}{x_0}$ for some $c \in (0,1]$, $x_0 \geq 0$ and $\alpha >0$. Then there is $C \geq 1$ such that for all $x>x_0$,
	\[
	Cx^2\vphi''(x) \leq K(1/x) \leq ex^2\vphi''(x).
	\]
\end{corollary}
Before embarking on our main results let us prove one key Lemma which provides control on the real part of the holomorphic extension of the Laplace exponent.  
\begin{lemma}\label{lem:1}
	Suppose that $\vphi'' \in \WLSC{\alpha-2}{c}{x_0}$ for some $c \in (0,1]$, $x_0 \geq 0$ and $\alpha >0$. Then there exists $C>0$ such that for all $w>x_0$ and $\lambda \in \RR$,
	\[
	\Re \big( \vphi(w)-\vphi(w+i\lambda) \big) \geq C\lambda^2\big( \vphi''(|\lambda| \vee w) \big).
	\]
\end{lemma}
\begin{proof}
	By the integral representation \eqref{eq:16}, for $\lambda \in \RR$ we have
	\[
	\Re \big( \vphi(w)-\vphi(w+i\lambda) \big) = \sigma^2\lambda^2+ \int_{(0,\infty)} \big( 1-\cos \lambda s \big) e^{-ws}\,\nu(ds).
	\]
	In particular, we see that the expression above is symmetric in $\lambda$. Thus, it is sufficient to consider $\lambda > 0$. Moreover, we infer that
	\begin{align}
	\Re \big( \vphi(w)-\vphi(w+i\lambda) \big)  \gtrsim \lambda^2 \bigg(\sigma^2+ \int_{(0,1/\lambda)} s^2 e^{-ws} \,\nu(d s)\bigg). \label{eq:4}
	\end{align}
	Due to Lemma \ref{lem:2} we obtain, for $\lambda \geq w$,
	\[
	\Re \big( \vphi(w)-\vphi(w+i\lambda) \big) \gtrsim \lambda^2 \bigg(\sigma^2+\int_{(0,1/\lambda)} s^2 \,\nu(ds)\bigg) \gtrsim \lambda^2\vphi''(\lambda).
	\]
	If $w>\lambda>0$, then by \eqref{eq:4},
	\begin{align*}
	\Re \big( \vphi(w)-\vphi(w+i\lambda) \big) &\gtrsim \lambda^2 \bigg(\sigma^2+\int_{(0,1/w)}s^2e^{-ws}\,\nu(ds)\bigg) \geq e^{-1}\lambda^2\bigg(\sigma^2+\int_{(0,1/w)} s^2\,\nu(ds)\bigg),
	\end{align*}
	which together with Lemma \ref{lem:2} ends the proof.
\end{proof}

\section{Asymptotics}\label{sec:asymptotics}

\begin{theorem}\label{thm:1}
	Let $\bfX$ be a spectrally positive L\'{e}vy processes of unbounded variation with the Laplace exponent $\vphi$. Suppose that $\vphi'' \in \WLSC{\alpha-2}{c}{x_0}$ for some $c \in (0,1]$, $x_0 \geq 0$ and $\alpha > 0$. Then the probability distribution of $X_t$ is absolutely continuous for all $t>0$. If we denote its density by $p(t,\cdot)$, then for each $\epsilon>0$ there is $M_0>0$ such that
	\[
	\bigg\lvert p \big( t,-t\vphi'(w) \big) \sqrt{ 2\pi t\vphi''(w)} \exp \Big\lbrace t \big( w\vphi'(w)-\vphi(w) \big) \Big\rbrace -1 \bigg\rvert \leq \epsilon,
	\]
	provided that $w>x_0$ and $tw^2\vphi''(w)>M_0$.
\end{theorem}

\begin{remark}\label{rem:1} 
Suppose $\sigma=0$. It is known (see e.g. \cite[Lemma 2.9]{GS2019}) that the scaling property with the scaling index $\alpha \geq 1$ implies unbounded variation. Theorem \ref{thm:1}, however, holds in greater generality. Namely, there are processes of unbounded variation which satisfy scaling condition with $\alpha$ strictly smaller than $1$. Indeed, proceeding exactly as in the \cite[Remark 4.12]{GLT2018} one may construct a L\'{e}vy measure which satisfies \eqref{eq:34} and whose corresponding second derivative of the Laplace exponent has lower and upper Matuszewska indices of order $-\frac32$ and $-\frac12$, respectively. Thus, the lower scaling condition for $\vphi''$ holds only for $\alpha<1$. We also note that the Gaussian component is not excluded.
\end{remark}
\begin{proof}
	Let $x=-t\vphi'(w)$ and $M>0$. We first show that
	\begin{equation}\label{eq:1}
	p(t,x) = \frac{1}{2\pi} \cdot \frac{e^{-\Theta(x/t,0)}}{\sqrt{t\vphi''(w)}}\int_{\RR} \exp \Bigg\lbrace -t \Bigg( \Theta \bigg( \frac{x}{t},\frac{u}{\sqrt{t\vphi''(w)}} \bigg) - \Theta \bigg( \frac{x}{t},0 \bigg) \Bigg) \Bigg\rbrace \, d u,
	\end{equation}
	provided that $w>x_0$ and $tw^2\vphi''(w)>M$, where for $\lambda>0$ we have set
	\begin{equation}\label{eq:2}
	\Theta(x/t,\lambda) = - \big( \vphi(w+i\lambda)+\frac{x}{t}(w+i\lambda) \big).
	\end{equation}
	To this end, we recall that, by the Mellin's inversion formula, if the limit
	\begin{equation}\label{eq:3}
	\lim_{L \to \infty} \frac{1}{2\pi i} \int_{w-iL}^{w+iL} e^{t\vphi(\lambda)+\lambda x} \, d \lambda \quad \text{exists},
	\end{equation}
	then the probability distribution of $X_t$ has a density $p(t,\cdot)$ and
	\[
	p(t,x) = \lim_{L \to \infty} \frac{1}{2\pi i} \int_{w-iL}^{w+iL} e^{t\vphi(\lambda)+\lambda x} \, d \lambda.
	\]
	By the change of variables we obtain
	\begin{align*}
	\frac{1}{2\pi i} \int_{w-iL}^{w+iL} e^{t\vphi(\lambda)+\lambda x}\,d\lambda &= \frac{1}{2\pi}\int_{-L}^L e^{-t\Theta(x/t,\lambda)} \, d \lambda \\ &= \frac{e^{-t\Theta(x/t,0)}}{2\pi} \int_{-L}^L \exp\Big\lbrace-t \big( \Theta(x/t,\lambda) -  \Theta(x/t,0) \big)\Big\rbrace \, d \lambda \\ &= \frac{e^{-t\Theta(x/t,0)}}{2\pi\sqrt{t\vphi''(w)}} \int_{-L\sqrt{t\vphi''(w)}}^{L\sqrt{t\vphi''(w)}} \exp \Bigg\lbrace -t \Bigg( \Theta \bigg( \frac{x}{t},\frac{u}{\sqrt{t\vphi''(w)}} \bigg) - \Theta \bigg( \frac{x}{t},0 \bigg) \Bigg) \Bigg\rbrace \, d u.
	\end{align*}
	Next, we observe that there is $C>0$, not depending on $M$, such that for all $u \in \RR$
	\begin{equation}\label{eq:6}
	t\Re \Bigg( \Theta \bigg( \frac{x}{t},\frac{u}{\sqrt{t\vphi''(w)}} \bigg) - \Theta \bigg( \frac{x}{t},0 \bigg) \Bigg) \geq C \Big( u^2 \wedge \big(|u|^{\alpha}M^{1-\alpha/2}\big) \Big), 
	\end{equation}
	provided that $w>x_0$ and $tw^2\vphi''(w)>M$. Indeed, by \eqref{eq:2} and Lemma \ref{lem:1}, for $w>x_0$ we get
	\begin{equation*}
	t\Re \Bigg( \Theta \bigg( \frac{x}{t},\frac{u}{\sqrt{t\vphi''(w)}} \bigg) - \Theta \bigg( \frac{x}{t},0 \bigg) \Bigg) \gtrsim \frac{|u|^2}{\vphi''(w)} \vphi'' \bigg( \frac{|u|}{\sqrt{t\vphi''(w)}} \vee w \bigg),
	\end{equation*}
	and \eqref{eq:6} follows by scaling property of $\vphi''$.
	Hence, \eqref{eq:3} follows from the dominated convergence theorem. Consequently, Mellin's inversion formula yields \eqref{eq:1}.

	Next, we prove that for each $\epsilon>0$ there is $M_0>0$ such that
	\begin{equation}\label{eq:7}
	\Bigg\lvert \int_{\RR} \exp \Bigg\lbrace -t \Bigg( \Theta \bigg( \frac{x}{t},\frac{u}{\sqrt{t\vphi''(w)}} \bigg) - \Theta \bigg( \frac{x}{t},0 \bigg) \Bigg) \Bigg\rbrace \, d u - \int_{\RR}e^{-\frac12 u^2}\, d u \Bigg\rvert \leq \epsilon,
	\end{equation}
	provided that $w>x_0$ and $tw^2\vphi''(w)>M_0$. In view of \eqref{eq:6}, by taking $M_0>1$ sufficiently large, we get
	\begin{align}\label{eq:8}
	\Bigg\lvert \int_{|u|\geq M_0^{1/4}} \exp \Bigg\lbrace -t \Bigg( \Theta \bigg( \frac{x}{t},\frac{u}{\sqrt{t\vphi''(w)}} \bigg) - \Theta \bigg( \frac{x}{t},0 \bigg) \Bigg) \Bigg\rbrace \, d u \Bigg\rvert \leq \int_{|u|\geq M_0^{1/4}} e^{-C|u|^{\alpha}}\, d u \leq \epsilon,
	\end{align}
	and
	\begin{align}\label{eq:9}
	\int_{|u|\geq M_0^{1/4}}e^{-\frac12u^2}\, d u \leq \epsilon.
	\end{align}
	Next, let us observe that there is $C>0$ such that
	\begin{equation}\label{eq:10}
	\Bigg\lvert t\Bigg( \Theta \bigg( \frac{x}{t},\frac{u}{\sqrt{t\vphi''(w)}} \bigg) - \Theta \bigg( \frac{x}{t}, 0 \bigg) \Bigg) - \frac12 |u|^2 \Bigg\rvert \leq C |u|^3M_0^{-\frac12}.
	\end{equation}
	Indeed, since
	$\partial_{\lambda}\Theta( x/t,0 )=0$,
	by Taylor's formula we get
	\begin{align}
	\nonumber \Bigg\lvert t\Bigg( \Theta \bigg( \frac{x}{t},\frac{u}{\sqrt{t\vphi''(w)}} \bigg) - \Theta \bigg( \frac{x}{t}, 0 \bigg) \Bigg) - \frac12 |u|^2 \Bigg\rvert &= \Bigg\lvert \frac12 \partial_{\lambda}^2\Theta \bigg( \frac{x}{t},\xi \bigg) \frac{|u|^2}{\vphi''(w)}-\frac12|u|^2 \Bigg\rvert \\ &= \frac{|u|^2}{2\vphi''(w)} \big\lvert \vphi''(w)-\vphi''(w+i\xi) \big\rvert, \label{eq:11}
	\end{align}
	where $\xi$ is some number satisfying
	$|\xi| \leq \frac{|u|}{\sqrt{t\vphi''(w)}}$.
	We also have
	\begin{align*}
	\big\lvert \vphi''(w)-\vphi''(w+i\xi) \big\rvert &\leq \int_{(0,\infty)}s^2e^{-ws}\big\lvert e^{-i\xi s}-1 \big\rvert \,\nu(ds) \leq 2|\xi| \int_{(0,\infty)}s^3e^{-ws}\,\nu(ds) = 2|\xi|\big( -\vphi'''(w) \big).
	\end{align*}
	Since $\vphi''$ is doubling, by \cite[Proposition 2.1]{GLT2018}, for $w>x_0$,
	\[
	\vphi''(w) \gtrsim w \big( -\vphi'''(w) \big),
	\]
	which together with the estimate on $|\xi|$ yield
	\begin{align}
	\nonumber \big\lvert \vphi''(w)-\vphi''(w+i\xi) \big\rvert &\leq C \frac{|u|}{\sqrt{t\vphi''(w)}} \cdot \frac{\vphi''(w)}{w} \leq CM_0^{-\frac12}|u|\vphi''(w), \label{eq:13}
	\end{align}
	if only $tw^2\vphi''(w)>M_0$, proving \eqref{eq:10} through \eqref{eq:11}.
	Finally, since for any $z \in \CC$,
	\[
	\big\lvert e^z-1 \big\rvert \leq |z|e^{|z|},
	\]
	\eqref{eq:10} implies
	\begin{align*}
	\Bigg\lvert \int_{|u| < M_0^{1/4}} \exp \Bigg\lbrace -t \Bigg( \Theta \bigg( \frac{x}{t},\frac{u}{\sqrt{t\vphi''(w)}} \bigg) - \Theta \bigg( \frac{x}{t},0 \bigg) \Bigg) \Bigg\rbrace \, d u - \int_{|u|< M_0^{1/4}}e^{-\frac12 u^2}\, d u \Bigg\rvert \\ \leq CM_0^{-\frac12} \int_{|u|<M_0^{1/4}} \exp \bigg\lbrace -\frac12|u|^2+CM_0^{-\frac12}|u|^3 \bigg\rbrace |u|^3 \, d u \leq \epsilon,
	\end{align*}
	provided that $M_0$ is sufficiently large, which together with \eqref{eq:8} and \eqref{eq:9} complete the proof of \eqref{eq:7} and the theorem follows.
\end{proof}
\begin{remark}\label{rem:3}
	If $x_0=0$ then the constant $M_0$ in Theorem \ref{thm:1} depends only on $\alpha$ and $c$. If $x_0>0$ then it also depends on
	\[
	\sup_{x \in [x_0,2x_0]}\frac{x\big(-\vphi'''(x)\big)}{\vphi''(x)}.
	\]
\end{remark}
\begin{corollary}\label{cor:1}
	Suppose that $\vphi'' \in \WLSC{\alpha-2}{c}{x_0}$ for some $c \in (0,1]$, $x_0 \geq 0$ and $\alpha > 0$. Then there is $M_0>0$ such that
	\[
	p(t,x) \approx \frac{1}{\sqrt{t\vphi''(w)}} \exp \Big\lbrace -t\big( w\vphi'(w)-\vphi(w) \big) \Big\rbrace,
	\]
	uniformly on the set
	\[
	\big\lbrace (t,x) \in \RR_+ \times \RR \colon x<-t\vphi'(x_0) \text{ and } tw^2\vphi''(w)>M_0 \big\rbrace,
	\]
	where $w=(\vphi')^{-1}(-x/t)$.
\end{corollary}
\begin{corollary}\label{cor:6}
	Suppose that $\vphi'' \in \WLSC{\alpha-2}{c}{x_0}$ for some $c \in (0,1]$, $x_0 \geq 0$ and $\alpha > 1$. Assume also that $\vphi'(\theta_1)=0$. Then there is $M >0$ such that
	\[
	p(t,x) \approx \frac{1}{\sqrt{t\vphi''(w)}} \exp \Big\lbrace -t\big( w\vphi'(w)-\vphi(w) \big) \Big\rbrace,
	\]
	uniformly on the set
	\begin{equation}\label{eq:24}
	\big\lbrace (t,x) \in \RR_+ \times \RR \colon  -x\vphi^{-1}(1/t) > M \text{ and } 0 \leq t\vphi(x_0 \vee 2 \theta_0) \leq 1 \big\rbrace,
	\end{equation}
	where $w=(\vphi')^{-1}(-x/t)$.
\end{corollary}
\begin{remark}
	The condition $\vphi'(\theta_1)=0$ covers the case $\EE X_1 \in [0,\infty]$. For the case $\EE X_1>0$ it is not, however, optimal, because we do not treat positive $x$ which may be in the area from Corollary \ref{cor:1}. We also note that $\alpha=2$ is included.
\end{remark}
\begin{proof}
	We verify that for $(t,x)$ belonging to the set \eqref{eq:24} we have $w>x_0$ and $wt\vphi''(w)>M_0$, where $w=(\vphi')^{-1}(-x/t)$. Let $M\geq C_1C_2$, where $C_1$, $C_2$ are taken from Propositions \ref{prop:5} and \ref{prop:1}, respectively. By Propositions \ref{prop:5} and \ref{prop:1},
	\[
	\frac{-x}{t} > M\frac{1}{t\vphi^{-1}(1/t)} = M \frac{\vphi \big( \vphi^{-1}(1/t) \big)}{\vphi^{-1}(1/t)} \geq C_1 (C_1^{-1}C_2^{-1}M) \vphi' \big( \vphi^{-1}(1/t) \big) \geq \vphi' \big( C_1^{-1}C_2^{-1}M \vphi^{-1}(1/t) \big).
	\]
	Thus,
	\begin{equation}\label{eq:25}
	w=(\vphi')^{-1}(-x/t)>C_1^{-1}C_2^{-1}M\vphi^{-1}(1/t).
	\end{equation}
	In particular, $w>x_0 \vee 2\theta_0$. Next, by Corollary \ref{cor:2}, there is $c_1 \in (0,1]$ such that
	$tw^2\vphi''(w) \geq c_1t\vphi(w)$.	Since, by Proposition \ref{prop:2}, there is $c_2 \in (0,1]$ such that $\vphi \in \WLSC{\alpha}{c_2}{x_0 \vee \theta_0}$, we obtain
	\[
	t\vphi(w)  \geq c_2 \bigg( \frac{w}{\vphi^{-1}(1/t)} \bigg)^{\alpha}.
	\]
	In view of \eqref{eq:25}, we get that
	\[
	tw^2\vphi''(w) > c_1c_2\big(C_1^{-1}C_2^{-1}M\big)^{\alpha}>M_0,
	\]
	provided that $M$ is sufficiently large. Applying Theorem \ref{thm:1} yields the desired result. 
\end{proof}

\section{Upper and lower estimates on the density}\label{sec:estimates}
In this Section we always assume $\vphi'' \in \WLSC{\alpha-2}{c}{x_0}$ for some $c \in (0,1]$, $x_0 \geq 0$ and $\alpha >0$. By Theorem \ref{thm:1}, the probability distribution $X_t$ has a density $p(t,\cdot)$. Let us define
\[
\Phi(x) = x^2\vphi''(x), \quad x>0.
\]
Clearly, $\Phi \in \WLSC{\alpha}{c}{x_0}$. By $\Phi^{-1}$ we denote its right-sided inverse, i.e.
\[
\Phi^{-1}(s) = \sup \{ r>0 \colon \Phi^*(r)=s \},
\]
where
\[
\Phi^*(r) = \sup_{0 < s \leq r} \Phi(s). 
\]
Clearly, $\Phi^{-1}$ is non-decreasing. Similarly to $\psi^{-1}$, we have
\begin{equation}\label{eq:18}
\Phi^* \big( \Phi^{-1}(s) \big) = s, \qquad \Phi^{-1} \big( \Phi^*(s) \big) \geq s.
\end{equation}
Observe that since for all $x>0$ and $\lambda \geq 1$,
\begin{equation}\label{eq:72}
\Phi(\lambda x) \leq \lambda^2 \Phi(x),
\end{equation}
we obtain
\begin{equation}\label{eq:45}
\Phi^*(\lambda x) \leq \lambda ^2 \Phi^*(x).
\end{equation}
Furthermore, for any $r>0$ let $u$ be such that $\Phi^{-1}(r)=u$. Then by \eqref{eq:45} and \eqref{eq:18}, for any $\lambda \geq 1$,
\[
\Phi^{-1}(\lambda r) = \Phi^{-1}\big(\lambda \Phi^*(u)\big) \geq \Phi^{-1} \big( \Phi^*(\sqrt{\lambda}u) \big) \geq \sqrt{\lambda} u.
\]
Thus, for any $r>0$ and $\lambda \geq 1$,
\begin{equation}\label{eq:17}
\Phi^{-1}(\lambda r) \geq \sqrt{\lambda} \Phi^{-1}(r).
\end{equation}
Let us start with a following observation on the exponent in Theorem \ref{thm:1}.
\begin{proposition}\label{prop:9}
	Suppose $\vphi'' \in \WLSC{\alpha-2}{c}{x_0}$ for some $c \in (0,1]$, $x_0 \geq 0$ and $\alpha > 0$. There is $C>0$ such that for all $x>x_0$,
	\[
	x\vphi'(x)-\vphi(x) \leq C \Phi(x).
	\]
\end{proposition}
\begin{proof}
	Observe that
	\[
	\big( x\vphi'(x)-\vphi(x) \big) - \big( x_0\vphi'(x_0)-\vphi(x_0) \big) = \int_{x_0}^x \Phi(u) \frac{du}{u} = \int_{x_0/x}^1 \Phi(xu) \frac{du}{u}.
	\]
	Thus, by the weak lower scaling property of $\Phi$, for any $x_0/x<u\leq 1$,
	\[
	\big( x\vphi'(x)-\vphi(x) \big) - \big( x_0\vphi'(x_0)-\vphi(x_0) \big) \lesssim \Phi(x) \int_0^1u^{\alpha-1}\, d u,
	\]
	which yields the claim for the case $x_0=0$. If $x_0>0$ then one can use continuity and positivity of $\Phi$.
\end{proof}

\begin{proposition}\label{prop:6}
	Suppose $\vphi'' \in \WLSC{\alpha-2}{c}{x_0}$ for some $c \in (0,1]$, $x_0 \geq 0$ and $\alpha>0$. Then there is $C \geq 1$ such that for all $0<r<1/x_0$,
	\[
	K(r) \leq h(r) \leq CK(r).
	\]
\end{proposition}
\begin{proof}
	Since $K(r)\leq h(r)$, it is enough to show that there is $C \geq 1$ such that for all $0<r<1/x_0$,
	\[
	h(r) \leq C K(r).
	\] 
	In view of \eqref{eq:42}, we have
	\begin{equation}\label{eq:35}
	h(r)  = 2 \int_r^{1/x_0} K(s )\frac{ds}{s} + 2 \int_{1/x_0}^{\infty} K(s) \frac{ds}{s}.
	\end{equation}
	By Corollary \ref{cor:4} we have $K(r) \approx \Phi(1/r)$ for $0<r<1/x_0$, which together with the scaling property imply
	\[
	\int_r^{1/x_0} K(s) \frac{ds}{s} \lesssim K(r), \quad 0<r<1/x_0.
	\]
	That finishes the proof for the case $x_0=0$. If $x_0>0$ then we observe that we also have $K(r) \gtrsim 1$ for all $0<r<1/x_0$. Since the second term on the right-hand side of \eqref{eq:35} is constant, the proof is finished.
\end{proof}
In view of \eqref{eq:43}, Proposition \ref{prop:6} and Corollary \ref{cor:4} we have
\begin{equation}\label{eq:44}
\psi^*(x) \approx h(1/x) \approx K(1/x) \approx \Phi(x),
\end{equation} 
for all $x>x_0$. In particular, $\psi^* \in \WLSC{\alpha}{c}{x_0}$ for some $c \in (0,1]$. Furthermore, for all $x>x_0$,
\begin{align*}
\psi^*(x) \lesssim K(1/x) = x^2 \int_{(0,1/x)} s^2 \, \nu(ds) \lesssim \int_{(0,1/x)} \big( 1-\cos sx \big)\,\nu(ds). 
\end{align*}
Thus, for all $x>x_0$,
\begin{equation}\label{eq:48}
\psi^*(x) \lesssim \Re \psi(x).
\end{equation}
As a corollary, we present now the aforementioned equivalence between scaling property of the second derivative of the Laplace exponent and the real part of the characteristic exponent.
\begin{corollary}\label{cor:5}
	We have $\vphi'' \in \WLSC{\alpha-2}{c}{x_0}$ for some $c \in (0,1]$, $x_0 \geq 0$ and $\alpha>0$ if and only if $\Re \psi \in \WLSC{\alpha}{\tilde{c}}{x_0}$ for some $\tilde{c} \in (0,1]$.
\end{corollary}
\begin{proof}
	In view of \eqref{eq:44} and \eqref{eq:48}, it remains to prove the second implication in the corollary. We first prove that $\psi^* \in \WLSC{\alpha}{c_1}{x_0}$ for some $c_1 \in (0,1]$. Let $x\geq x_0$ and $\lambda \geq 1$. By scaling property of $\Re \psi$,
	\begin{align*}
	\psi^*(\lambda x) &= \max \big\lbrace \psi^*(\lambda x_0), \sup_{\lambda x_0<x\leq \lambda x} \Re \psi(r) \big\rbrace \\ &\gtrsim \max \big\lbrace \psi^*(x_0), \lambda^{\alpha} \sup_{x_0<r\leq x} \Re \psi(r) \big\rbrace.
	\end{align*}
	Now observe that since $\lim_{r \to \infty} \Re \psi(r)=\infty$, there is $x_1 \geq x_0$ such that $\Re \psi(x) \geq \psi^*(x_0)$ for all $x \geq x_1$ and consequently, for all $\lambda \geq 1$ and $x \geq x_1$,
	\[
	\psi^*(\lambda x) \gtrsim \max \big\lbrace \psi^*(x_0), \lambda^{\alpha} \sup_{r\leq x} \Re \psi(r) \big\rbrace = \lambda^{\alpha} \psi^*(x),
	\]
	and by standard extension argument we get scaling property of $\psi^*$ as desired.
	
	It remains to notice that by the integral representation of $\vphi''$,
	\[
	x^{-2}K(1/x) \lesssim \vphi''(x) \lesssim x^{-2}h(1/x).
	\]
	Thus, \cite[Lemma 2.3]{GS2019} yields the claim.
\end{proof}
\begin{proposition}\label{prop:7}
	Suppose $\vphi'' \in \WLSC{\alpha-2}{c}{x_0}$ for some $c \in (0,1]$, $x_0 \geq 0$ and $\alpha >0$. Then for all $r>2h(1/x_0)$,
	\begin{equation*}
	\frac{1}{h^{-1}(r)} \approx \psi^{-1}(r).
	\end{equation*}
	Furthermore, there is $C \geq 1$ such that for all $\lambda \geq 1$ and $r>2h(1/x_0)$,
	\[
	\psi^{-1}(\lambda r) \leq C\lambda^{1/\alpha} \psi^{-1}(r).
	\]
\end{proposition}
\begin{proof}
	Follows immediately by \cite[(5.1)]{GS2017} and scaling property of $h^{-1}$ provided by Proposition \ref{prop:6} and \cite[Lemma 2.3]{GS2019}.
\end{proof}
\begin{proposition}\label{prop:8}
	Suppose $\vphi'' \in \WLSC{\alpha-2}{c}{x_0}$ for some $c \in (0,1]$, $x_0 \geq 0$ and $\alpha>0$. Then for all $x>x_0$,
	\begin{equation}\label{eq:39}
	\psi^*(x) \approx \Phi^*(x),
	\end{equation}
	and for all $r>\Phi(x_0)$,
	\begin{equation}\label{eq:40}
	\psi^{-1}(r) \approx \Phi^{-1}(r).
	\end{equation}
	Furthermore, there is $C \geq 1$ such that for all $\lambda \geq 1$ and $r>\Phi(x_0)$,
	\begin{equation}\label{eq:85}
	\Phi^{-1}(\lambda r) \leq C \lambda^{1/\alpha} \Phi^{-1}(r). 
	\end{equation}
\end{proposition}
\begin{proof}
	The first inequality of \eqref{eq:39} follows immediately from \eqref{eq:44}. If $x_0=0$ then the second inequality is also the consequence of \eqref{eq:44}. If this is not the case then observe that for $x>x_0$,
	\begin{align*}
	\Phi^*(x) &= \max \big\lbrace \sup_{0<y\leq x_0} \Phi(y), \sup_{x_0<y\leq x} \Phi(y) \big\rbrace \\ &\lesssim \max \{ \Phi^*(x_0), \psi^*(x) \} \\ &\leq \bigg( 1 + \frac{\Phi^*(x_0)}{\psi^*(x_0)} \bigg) \psi^*(x),
	\end{align*}
	and the first part is proved. Furthermore, it follows that
	\[
	\psi^{-1}(C^{-1}r) \leq \Phi^{-1}(r) \leq \psi^{-1}(Cr),
	\]
	for all $r>C\psi^*(x_0)$. Hence, by Proposition \ref{prop:7},
	\[
	\Phi^{-1}(r) \approx \psi^{-1}(r),
	\]
	for all $r>C \max \{ \psi^*(x_0), 2h(1/x_0) \}$, and \eqref{eq:40} follows by standard extension argument. The scaling property of $\Phi^{-1}$ is an easy consequence of \eqref{eq:40} and Proposition \ref{prop:7}.
\end{proof}
Since $\Phi \leq \Phi^*$, by Proposition  \ref{prop:8} and \eqref{eq:44} we immediately obtain the following.
\begin{remark}\label{rem:2}
	Suppose $\vphi'' \in \WLSC{\alpha-2}{c}{x_0}$ for some $c \in (0,1]$, $x_0 \geq 0$ and $\alpha>0$. There is $c_1 \in (0,1]$ such that for all $x>x_0$,
	\[
	c_1\Phi^*(x) \leq \Phi(x) \leq \Phi^*(x).
	\]
\end{remark}
\begin{proposition}\label{prop:10}
	Suppose that $\vphi'' \in \WLSC{\alpha-2}{c}{x_0}$ for some $c \in (0,1]$, $x_0 \geq 0$ and $\alpha>1$. Assume also that $\vphi'(\theta_1)=0$. Then for all $x>x_0 \vee 2\theta_0$,
	\[
	\Phi^*(x) \approx \vphi(x),
	\]
	and for all $r>\Phi(x_0 \vee 2\theta_0)$,
	\[
	\Phi^{-1}(r) \approx \vphi^{-1}(r).
	\]
\end{proposition}

\begin{proof}
	Corollary \ref{cor:2} and Remark \ref{rem:2} yield the first part. The proof of the second omitted due to similarity to the proof of Proposition \ref{prop:8}.
\end{proof}

\subsection{Upper estimates}
From this moment on, we additionally assume that $\sigma=0$. As explained in Preliminaries, that is equivalent to saying that $\bfX$ satisfies the integral condition \eqref{eq:61}. Suppose $\vphi'' \in \WLSC{\alpha-2}{c}{x_0}$. Recall that since $\vphi''$ is positive and continuous on $(0,\infty)$, if $x_0>0$ then at the cost of worsening the constant $c$, we can extend area of comparability to any $x_1 \in (0,x_0)$ so that $\vphi'' \in \WLSC{\alpha-2}{c'}{x_1}$, where $c'$ depends on $x_1$. Thus, if $\theta_1>0$ and $x_0>0$ then we may and do assume that $x_0$ is shifted so that $x_0 \leq \theta_1$. With that in mind, let us define $\eta \colon [0,\infty) \mapsto [0,\infty]$,
\begin{align*}
\eta(s) = \begin{cases}
\infty & \text{if } s=0,\\
s^{-1}\Phi^*(1/s) & \text{if } 0<s< x_0^{-1},\\
As^{-1}\big\lvert \vphi(1/s) \big\rvert & \text{if } x_0^{-1}\leq s,
\end{cases}
\end{align*}
where $A=\Phi^*(x_0)/\big\lvert \vphi(x_0) \big\rvert$.

Let us comment on the function $\eta$. As we will see in Theorem \ref{thm:2}, it will play a role of majorant on the transition density. In such setting it is clear that $\eta$ must be non-negative. Moreover, in the proof we will require from it to be monotone. Now, if $\theta_1>0$ then we know that $\vphi$ is indeed monotone for $x \in (0,\theta_1)$, but also negative. Thus, a change of sign is required. On the other hand, if $\theta_1=0$ then $\vphi \geq 0$ and there is no need for absolute value and shifting of $x_0$. In general, however, $\vphi$ may be negative in a neighbourhood of the origin and change sign in $\theta_1$, so one has to be careful in expanding scaling area to the proper place. Note that by Corollary \ref{cor:3} and Remark \ref{rem:2}, $A \leq c'$, where $c'$ depends only on $\theta_1$.

Denote
\begin{equation}\label{eq:87}
b_r =b + \int_{(0,\infty)} s \big( \ind{s<r} - \ind{s<1} \big)\,\nu(ds).
\end{equation}

\begin{theorem}\label{thm:2}
	Let $\bfX$ be a spectrally positive L\'{e}vy process of infinite variation with the L\'{e}vy-Khintchine exponent $\psi$ and the Laplace exponent $\vphi$. Suppose that $\sigma=0$ and $\vphi'' \in \WLSC{\alpha-2}{c}{x_0}$ for some $c \in (0,1]$, $x_0 \geq 0$ and $\alpha >0$. We also assume that $\nu(d x)$ has an almost monotone density $\nu(x)$. Then the probability distribution of $X_t$ has a density $p(t,\cdot)$. Moreover, there is $C > 0 $ such that for all $t \in (0,1/\Phi(x_0))$ and $x \in \RR$,
	\begin{equation}\label{eq:28}
	p\big( t,x+tb_{1/\psi^{-1}(1/t)}\big) \leq C \min \big\lbrace \Phi^{-1}(1/t),t\eta(|x|) \big\rbrace.
	\end{equation}
\end{theorem}
\begin{proof}
	In the first step we verify assumptions of \cite[Theorem 5.2]{GS2017}. First observe that for any $\lambda >0$,
	\begin{equation}\label{eq:36}
	\vphi''(\lambda) \geq \int_0^{1/\lambda} s^2 e^{-\lambda s} \nu(s) \, d s \gtrsim \nu(1/\lambda)\lambda^{-1},
	\end{equation}
	thus, by Corollary \ref{cor:3}, $\nu(x) \lesssim \eta(x)$ for all $x>0$. Since $\eta$ is non-increasing,  we conclude that the first assumption is satisfied.	Next, we claim that $\eta$ has a doubling property on $(0,\infty)$. Indeed, since $\vphi''$ is non-increasing, by Corollary \ref{cor:4} and \eqref{eq:44}, for $0<s<x_0^{-1}$,
	\[
	\eta \big( \tfrac{1}{2}s \big) \approx  4s^{-2}\vphi'' ( 2/s ) \lesssim s^{-2}\vphi''(s) \approx \eta(s).   
	\]
	This completes the argument for the case $x_0=0$. If $x_0>0$, then Proposition \ref{prop:1} (or \eqref{eq:47} if $\theta_0>0$) yields the claim for $s>2x_0^{-1}$. Lastly, the function
	\[
	\big[ \tfrac12x_0,x_0 \big] \ni x \mapsto \frac{\Phi^*(2x)}{\big\lvert \vphi(x) \big\rvert},
	\]
	is continuous, hence bounded.
	
	Therefore, since $s \wedge |x|-\tfrac12 |x| \geq \tfrac12s$ for $s>0$ and $x \in \RR$, the doubling property of $\eta$ and \eqref{eq:44} imply the second assumption. Finally, since $\psi^*$ has the weak lower scaling property and satisfies \eqref{eq:48}, by \cite[Theorem 3.1]{GS2019} and Proposition \ref{prop:7}, there are $C>0$ and $t_1 \in (0,\infty]$ such that for all $t \in (0,t_1)$,
	\begin{equation*}
	\int_{\RR} e^{-t\Re \psi(\xi)}\, d \xi \leq C \psi^{-1}(1/t),
	\end{equation*}
	with $t_1=\infty$ whenever $x_0=0$. Note that if $t_1<48/\Phi(x_0)$, then using positivity and monotonicity we can expand the estimate for $t_1 \leq t <48/\Phi(x_0)$, and the first step is finished.
		
	Therefore, by \cite[Theorem 5.2]{GS2017} there is $C>0$ such that for all $t \in (0,1/\Phi(x_0))$ and $x \in \RR$,
	\[
	p \big( t,x+tb_{1/\psi^{-1}(1/t)} \big) \leq C \psi^{-1}(1/t) \cdot \min \Big\lbrace 1,t\big(\psi^{-1}(1/t) \big)^{-1}\eta(|x|)+\big( 1+|x|\psi^{-1}(1/t) \big)^{-3}\Big\rbrace.
	\]
	Now, it suffices to prove that
	\begin{equation}\label{eq:27}
	\frac{\psi^{-1}(1/t)}{\big( 1+|x|\psi^{-1}(1/t) \big)^3} \lesssim t\eta(|x|),
	\end{equation}
	whenever $t\eta(|x|) \leq \frac{A}{c'} \Phi^{-1}(1/t)$.

	First let us observe that for any $\epsilon \in (0,1]$, the condition $t\eta(|x|) \leq \frac{A\epsilon}{c'}\Phi^{-1}(1/t)$ implies
	\begin{equation}\label{eq:26}
	t\Phi^*(1/|x|) \leq \epsilon|x|\Phi^{-1}(1/t).
	\end{equation}
	Indeed, by Corollary \ref{cor:3} and Remark \ref{rem:2}, we have $|x|\eta(|x|) \geq \frac{A}{ c'}\Phi^*(1/|x|)$, which entails \eqref{eq:26}.	Furthermore, we have $\epsilon^{1/3}|x|\Phi^{-1}(1/t) \geq 1$, since otherwise by \eqref{eq:45} we would have
	\[
	1<t\Phi^*\bigg(\frac{1}{\epsilon^{1/3}|x|}\bigg) \leq \frac{1}{\epsilon^{2/3}} t\Phi^*(1/|x|),	
	\]
	which in turn yields
	\[
	\epsilon^{1/3}|x|\Phi^{-1}(1/t) < \epsilon^{-2/3}t\Phi^*(1/|x|),	
	\]
	contrary to \eqref{eq:26}.

	Now we suppose $t\eta(|x|)\leq \frac{A}{c'}\Phi^{-1}(1/t)$. Since $|x|\Phi^{-1}(1/t) \geq 1$, by \eqref{eq:45} we infer that
	\[
	t\Phi^*(1/|x|) = \frac{\Phi^*(1/|x|)}{\Phi^* \big( |x|\Phi^{-1}(1/t) \cdot 1/|x| \big)} \geq \frac{1}{\big( |x|\Phi^{-1}(1/t) \big)^2} \geq \frac{|x|\Phi^{-1}(1/t)}{\big(1+ |x|\Phi^{-1}(1/t) \big)^3},
	\]
	It remains to notice that Proposition \ref{prop:8}
	entails \eqref{eq:27}, and the proof is completed.
\end{proof}

\begin{remark}\label{rem:4}
	In statement of Theorem \ref{thm:2} we may replace $b_{1/\psi^{-1}(1/t)}$ by $b_{1/\Phi^{-1}(1/t)}$. Indeed, if $0<r_1\leq r_2 <1/\Phi(x_0)$ then by \eqref{eq:44} and Proposition \ref{prop:8},
	\[
	\big\lvert b_{r_1}-b_{r_2} \big\rvert \leq \int_{(r_1,r_2]} s\,\nu(ds) \leq r_1^{-1}r_2^2h(r_2) \lesssim r_1^{-1}r_2^2 \psi^*(r_2^{-1}) \lesssim r_1^{-1}r_2^2 \Phi^*(r_2^{-1}).
	\]
	Thus, again by Proposition \ref{prop:8}, there is $C \geq 1$ such that for all $t \in (0,1/\Phi(x_0))$,
	\[
	\Big\lvert b_{1/\psi^{-1}(1/t)} - b_{1/\Phi^{-1}(1/t)} \Big\rvert \leq \frac{C}{t\Phi^{-1}(1/t)}.
	\]
	Now, recall that if $t\eta(|x|) \leq \frac{A\epsilon}{c'}\Phi^{-1}(1/t)$, then by the proof of Theorem \ref{thm:2} we have $|x|\Phi^{-1}(1/t) \geq \epsilon^{-1/3}$.
	Therefore, by taking $\epsilon = (2C)^{-3}$ we obtain $|x| \geq \frac{2C}{\Phi^{-1}(1/t)}$ and
	by monotonicity and doubling property of $\eta$ we conclude that
	\[
	\eta \bigg( \Big\lvert x+t \Big( b_{1/\psi^{-1}(1/t)} - b_{1/\Phi^{-1}(1/t)} \Big) \Big\rvert \bigg) \lesssim \eta (|x|).
	\]
\end{remark}

\subsection{Lower estimates}
We begin with an estimate which, together with Theorem \ref{thm:2} will allow us to localize the supremum of $p(t,\cdot)$. Note that here we require the scaling condition with $\alpha \geq 1$.
\begin{lemma}\label{lem:3}
	Let $\bfX$ be a spectrally positive L\'{e}vy process of infinite variation and Laplace exponent $\vphi$. Suppose that $\sigma=0$ and $\vphi'' \in \WLSC{\alpha-2}{c}{x_0}$ for some $c \in (0,1]$, $x_0 \geq 0$ and $\alpha \geq 1$. Then there is $M_0>1$ such that for each $M\geq M_0$ and $\rho_1,\rho_2>0$ there exists $C>0$, so that for all $t \in (0,1/\Phi(x_0))$ and any $x \in \RR$ satisfying
	\[
	-\frac{\rho_1}{\Phi^{-1}(1/t)} \leq x+t\vphi' \big( \Phi^{-1}(M/t) \big) \leq \frac{\rho_2}{\Phi^{-1}(1/t)}
	\] 
	we have
	\[
	p(t,x) \geq C\Phi^{-1}(1/t).
	\]
\end{lemma}
\begin{proof}
	Without loss of generality we may assume that $b=0$. By \cite[Theorem 5.4]{GS2019}, for any $\theta>0$ there is $c>0$ such that for all $t \in (0,1/\Phi(x_0))$ and $|x| \leq \theta h^{-1}(1/t)$,
	\[
	p \big( t,x+tb_{h^{-1}(1/t)} \big) \geq c \big( h^{-1}(1/t) \big)^{-1}.
	\]
	Since by Propositions \ref{prop:7} and \ref{prop:8}, we have $h^{-1}(1/t) \approx 1/\Phi^{-1}(1/t)$ for all $t \in (0,1/\Phi(x_0))$, it suffices to prove that
	\[
	\Big| t\vphi' \big( \Phi^{-1}(M/t) \big) + tb_{h^{-1}(1/t)} \Big| \leq \frac{c_1}{\Phi^{-1}(1/t)},
	\]
	for some $c_1>0$. To this end observe that for $\lambda_1,\lambda_2>0$ we have
	\begin{align*}
	\big| \vphi'(\lambda_1) + b_{\lambda_2} \big| &= \bigg| \int_0^{\infty} s \big( \ind{s<\lambda_2}-e^{-\lambda_1 s} \big) \,\nu({\rm d}s) \bigg| \lesssim \lambda_1 \int_0^{\lambda_2} s^2\,\nu({\rm d}s) + \int_{\lambda_2}^{\infty} se^{-\lambda_1 s}\,\nu({\rm d}s) \\ &\lesssim \lambda_1\lambda_2^2  K(\lambda_2) + \lambda_1^{-1}h(\lambda_2).
	\end{align*}
	Now put $\lambda_1 = \Phi^{-1}(M/t)$ and $\lambda_2 = h^{-1}(1/t)$. Then using again Propositions \ref{prop:7} and \ref{prop:8} we infer that
	\[
	\big| \vphi'(\lambda_1) + b_{\lambda_2} \big| \lesssim \frac{1}{t\Phi^{-1}(1/t)},
	\]
	and the proof is completed.
\end{proof}

Now we treat the right tail of the transition density. In general, based on results concerning various kinds of L\'{e}vy processes, we expect the decay to be expressed in terms of L\'{e}vy measure. For instance, in the case of unimodal L\'{e}vy processes satisfying some scaling conditions it is known \cite[Theorem 21 and Corollary 23]{BGR14} that
\[
p(t,x) \approx p(t,0) \wedge t\nu(x).
\]
In the non-symmetric case, a similar right tail decay is displayed by transition densities of subordinators (see e.g. \cite{GLT2018}). This is also the case for spectrally one-sided L\'{e}vy processes, as the following lemma states. See also Theorem \ref{thm:4} and Remark \ref{rem:5} for comments on two-sided estimate.
  
\begin{lemma}\label{lem:4}
	Let $\bfX$ be a spectrally positive L\'{e}vy process of infinite variation with Laplace exponent $\vphi$. Suppose that $\sigma=0$ and $\vphi'' \in \WLSC{\alpha-2}{c}{x_0}$ for some $x_0 \geq 0$, $c \in (0,1]$ and $\alpha > 1$. We also assume that the L\'{e}vy measure $\nu(dx)$ has an almost monotone density $\nu(x)$. Then the probability distribution of $\bfX$ has a density $p(t,\cdot)$. Moreover, there are $M_0>1$, $\rho_0>0$ and $C>0$ such that for all $t \in (0,1/\Phi(x_0))$ and
	\[
	x \geq \frac{\rho_0}{\Phi^{-1}(1/t)},
	\]
	we have
	\[
	p(t,x) \geq Ct\nu(x).
	\]
\end{lemma}
\begin{proof}
	Without loss of generality we may and do assume that $b=0$. Let $\lambda>0$. We decompose the L\'{e}vy measure $\nu(dx)$ as follows: let $\nu_1(dx)=\nu_1(x)\, d x$ and $\nu_2(dx)=\nu_2(x)\, d x$, where
	\begin{equation}\label{eq:67}
	\nu_1(x) = \frac12\nu(x) \ind{[\lambda,\infty)}(x) \qquad \text{and} \qquad \nu_2(x)=\nu(x)-\nu_1(x).
	\end{equation}
	For $u>0$ we set
	\begin{align}
	\vphi_1(u) &= \int_0^{\infty} \big( e^{-us}-1\big)\nu_1(s) \, d s, \nonumber \\
	\vphi_2(u) &= \int_0^{\infty} \big( e^{-us}-1+us\ind{s<1} \big)\nu_2(s)\, d s + u \int_0^1 s\nu_1(s)\, d s = \vphi(u)-\vphi_1(u).  \label{eq:66}
	\end{align}
	Let $\bfX^{(j)}$ be a spectrally positive L\'{e}vy processes having the Laplace exponent $\vphi_j$, $j \in \{1,2\}$. First we observe that $\tfrac12 \nu \leq \nu_2 \leq \nu$, thus, for every $u>0$,
	\[
	\frac12 \vphi''(u) \leq \vphi_2''(u) \leq \vphi''(u),
	\]
	and consequently,
	\begin{equation}\label{eq:68}
	\frac12 \Phi(u) \leq \Phi_2(u) \leq \Phi(u).
	\end{equation}
	In particular, since $\phi'' \in \WLSC{\alpha-2}{c}{x_0}$, by Theorem \ref{thm:1}, random variables $X_t$ and $X_t^{(2)}$ are absolutely continuous for all $t>0$. By $p(t,\cdot)$ and $p^{(2)}(t,\cdot)$ we denote its densities. Observe that $\bfX^{(1)}$ is in fact a compound Poisson process. If we denote its probability distribution by $P_t^{(1)}(dx)$, then by \cite[Remark 27.3]{Sato},
	\begin{equation}\label{eq:73}
	P_t^{(1)}(dx) \geq te^{-\nu_1(\RR)}\nu_1(x)\, d x.
	\end{equation}
	Note that if $\lambda \geq c_1/\Phi^{-1}(1/t)$ for some $c_1>0$, then by \eqref{eq:44},
	\begin{align}
	t\nu_1(\RR) &= \frac12t\int_{\lambda}^{\infty} \nu(x)\, d x \leq \frac12th\big( c_1/\Phi^{-1}(1/t) \big) \lesssim t h \big( 1/\Phi^{-1}(1/t) \big) \lesssim 1, \label{eq:74}
	\end{align}
	where in the third line we use \cite[Lemma 2.1]{GS2019}.	
	
	Now, denote
	\[
	x_t = -t\vphi_2' \big( \Phi_2^{-1}(M_0/t) \big),
	\]
	where $M_0$ is taken from Lemma \ref{lem:3}. We claim that there is $\rho_0>0$ such that for all $t \in (0,1/\Phi(x_0))$,
	\begin{equation}\label{eq:65}
		\frac{\rho_0}{\Phi^{-1}(1/t)}\geq -x_t.
	\end{equation}
	Indeed, observe that by \eqref{eq:68}, for any $s>0$,
	\begin{equation}\label{eq:59}
	\Phi_2^{-1}(s) \geq \Phi^{-1}(s).
	\end{equation}
	Thus, using Proposition \ref{prop:3} and monotonicity of $\Phi^{-1}$ we conclude that there is $c_2>0$ such that
	\[
	t\vphi_2' \big( \Phi_2^{-1}(M_0/t) \big) \leq c_2t \frac{M_0}{t} \frac{1}{\Phi_2^{-1}(M_0/t)} \leq \frac{c_2M_0}{\Phi^{-1}(M_0/t)} \leq \frac{c_2M_0}{\Phi^{-1}(1/t)},
	\]
	and \eqref{eq:65} follows with $\rho_0=c_2M_0$.
	
	Now, we apply Lemma \ref{lem:3} for $\bfX^{(2)}$. For all $\rho>0 $ there is $C>0$ such that for all $t \in (0,1/\Phi_2(x_0))$ and $x \in \RR$ satisfying
	\[
	x_t-\frac{\rho}{\Phi_2^{-1}(1/t)} \leq x \leq x_t+\frac{\rho}{\Phi_2^{-1}(1/t)},
	\]
	we have
	\[
	p^{(2)}(t,x) \geq C \Phi_2^{-1}(1/t).
	\]
	Note that if $x_0>0$ then we may easily extend the above for $t \in (0,1/\Phi(x_0))$. Let $\rho_0$ be taken from \eqref{eq:65} and set $\lambda = x_t+\frac{\rho}{\Phi_2^{-1}(1/t)}$,
	where $\rho=\frac32 \rho_0$. Then it follows that $\lambda \geq \frac12 \frac{\rho_0}{\Phi_2^{-1}(1/t)}$,
	and consequently,
	\begin{equation}\label{eq:75}
	\int_0^{\lambda} p^{(2)} (t,x) \, d x \gtrsim 1.
	\end{equation}
	Thus, using \eqref{eq:73} and \eqref{eq:74}, for $x \geq 2\lambda$ we have
	\begin{align*}
	p (t,x) &= \int_{\RR} p^{(2)}(t,x-y) P_t^{(1)}(dy) \gtrsim t\int_{\RR} p^{(2)}(t,x-y)\nu_1(y) \, d y \geq \frac12 t\int_{\lambda}^x p^{(2)} (t,x-y)\nu(y)\, d y.
	\end{align*}
	Finally, using almost monotonicity of $\nu$ and \eqref{eq:75} we get
	\begin{align*}
	p (t,x) &\gtrsim t\nu(x) \int_0^{\lambda} p^{(2)} (t,y) \, d y \gtrsim t\nu(x). 
	\end{align*}
	Now it remains to observe that by \eqref{eq:66}, for any $u>0$,
	\[
	\vphi_2'(u) \geq \vphi'(u).
	\]
	Thus, by \eqref{eq:59},
	\begin{align*}
	\lambda = -t\vphi_2' \big( \Phi_2^{-1}(M_0/t) \big) + \frac{\rho}{\Phi^{-1}(1/t)} \leq -t\vphi' \big( \Phi^{-1}(M_0/t) \big) + \frac{\rho}{\Phi^{-1}(1/t)}.
	\end{align*}
	The proof is completed.
\end{proof}

\section{Sharp two-sided estimates}\label{sec:sharp}
This Section is devoted to derivation of sharp two-sided estimates. As mentioned in the Introduction, we will require here the upper scaling condition as well, in order to express the L\'{e}vy density in terms of Laplace exponent $\vphi$. First, however, thanks to strict separation from the limit case $\alpha=1$, we are able to provide simpler expression for the localization of $\sup_{x \in \RR} p(t,x)$.
\begin{theorem}\label{thm:3}
	Let $\bfX$ be a spectrally positive L\'{e}vy process of infinite variation with Laplace exponent $\vphi$. Suppose that $\sigma=0$ and $\vphi \in \WLSC{\alpha}{c}{x_0}$ for some $c \in (0,1]$, $x_0 \geq 0$ and $\alpha > 1$. We assume also that $\vphi'(\theta_1)=0$. Then for all $-\infty<\chi_1<\chi_2<\infty$ there is $C>0$ such that for all $t \in (0,1/\Phi(x_0 \vee 2\theta_1))$ and $x \in \RR$ satisfying
	\begin{equation}\label{eq:64}
	\chi_1 < x\vphi^{-1}(1/t) < \chi_2,
	\end{equation}
	we have
	\[
	C^{-1}\vphi^{-1}(1/t) \leq p(t,x) \leq C\vphi^{-1}(1/t).
	\]
\end{theorem}

\begin{proof}
	First, let us note that by Proposition \ref{prop:10}, there is $C' \geq 1$ such that for all $r \in (0,1/\Phi(x_0 \vee 2\theta_0))$,
	\begin{equation}\label{eq:76}
	C'^{-1} \Phi^{-1}(r) \leq \vphi^{-1}(r) \leq C' \Phi^{-1}(r).
	\end{equation}
	Thus, in view of \cite[Theorem 3.1]{GS2019} and Propositions \ref{prop:7} and \ref{prop:8}, it is enough to prove the first inequality in \eqref{eq:64}. Next, we observe that assumptions of Lemma \ref{lem:3} are satisfied. Let $M_0$ be taken from Lemma \ref{lem:3}; for fixed $M>M_0$ and $t \in (0,1/\Phi(x_0 \vee 2\theta_1))$ we set
	\[
	x_t = -t\vphi' \big( \Phi^{-1}(M/t) \big).
	\]
	By Propositions \ref{prop:5} and \ref{prop:8}, and \eqref{eq:76}, there is $c_1 \in (0,1]$ such that
	\[
	t\vphi' \big( \Phi^{-1}(M/t) \big) \geq \frac{c_1}{\vphi^{-1}(1/t)}.
	\]
	Furthermore, by Proposition \ref{prop:3}, \eqref{eq:17} and \eqref{eq:76}, there is $C_1 \geq 1$ such that
	\[
	t\vphi' \big( \Phi^{-1}(M/t) \big) \leq \frac{C_1}{\vphi^{-1}(1/t)}.
	\]
	Now we apply Lemma \ref{lem:3}. Pick $\rho_1$ and $\rho_2$ so that
	\[
	-c_1-\frac{\rho_1}{C'} \leq \chi_1 \qquad \text{and} \qquad -C_1+\frac{\rho_2}{C'} \geq \chi_2.
	\]
	Then is is clear that
	\[
	\bigg[ \frac{\chi_1}{\vphi^{-1}(1/t)}, \frac{\chi_2}{\vphi^{-1}(1/t)}\bigg] \subset \bigg( x_t-\frac{\rho_1}{\Phi^{-1}(1/t)}, x_t+\frac{\rho_2}{\Phi^{-1}(1/t)} \bigg).
	\]
	Hence, by Lemma \ref{lem:3} and \eqref{eq:76}, for all $t \in (0,1/\Phi(x_0 \vee 2\theta_0))$ and $x \in \RR$ satisfying
	\[
	\chi_1 \leq x\vphi^{-1}(1/t) \leq \chi_2
	\]
	we have
	\[
	p(t,x) \gtrsim \vphi^{-1}(1/t),
	\]
	and the theorem follows.
\end{proof}
Proceeding exactly as in the proof of \cite[Proposition 4.15]{GLT2018} and applying Corollary \ref{cor:2} yields the following.
\begin{proposition}\label{prop:11}
	Assume that the L\'{e}vy measure $\nu(dx)$ has an almost monotone density $\nu(x)$. Suppose that $\vphi'(\theta_1)=0$ and $\vphi \in \WLSC{\alpha}{c}{x_0} \cap \WUSC{\beta}{C}{x_0}$ for some $c \in (0,1]$, $C \geq 1$, $x_0 \geq \theta_0$ and $1 < \alpha \leq \beta <2$. Then there is $c' \in (0,1]$ such that for all $0<x<x_0^{-1} \wedge (2\theta_0)^{-1}$,
	\[
	\nu(x) \geq c'x^{-1}\vphi(1/x).
	\]
\end{proposition}

Now we are ready to prove our main result of this Section.
\begin{theorem}\label{thm:4}
	Let $\bfX$ be a spectrally positive L\'{e}vy process of infinite variation with the Laplace exponent $\vphi$ such that $\theta_1=0$ and $\vphi'(0)=0$. Suppose that $\sigma=0$ and $\vphi \in \WLSC{\alpha}{c}{x_0} \cap \WUSC{\beta}{C}{x_0}$ for some $c \in (0,1]$, $C \geq 1$, $x_0 \geq 0$, and $1<\alpha \leq \beta < 2$. We also assume that the L\'{e}vy measure $\nu(dx)$ has an almost monotone density $\nu(x)$. Then there is $x_1 \in (0,\infty]$ such that for all $t \in (0,1/\Phi(x_0))$ and $x \in (-\infty,x_1)$,
	\begin{align*}
	p(t,x) \approx \left\lbrace \begin{array}{lll}
	\big( t\vphi''(w) \big)^{-\frac12} \exp \big\lbrace -t \big( w\vphi'(w)-\vphi(w) \big) \big\rbrace, & \text{if } & x\vphi^{-1}(1/t) \leq -1,\\
	\vphi^{-1}(1/t), & \text{if } & -1 < x\vphi^{-1}(1/t) \leq 1,\\ 
	tx^{-1}\vphi(1/x), & \text{if } & x\vphi^{-1}(1/t)>1.
	\end{array}	\right.
	\end{align*}
	where $w = (\vphi')^{-1}(-x/t)$. If $x_0=0$ then $x_1=\infty$.
\end{theorem}
\begin{proof}
	Set $x_1=x_0^{-1}$. First we note that in view of Propositions \ref{prop:5} and \ref{prop:3}, $\vphi'' \in \WLSC{\alpha-2}{c}{x_0}$. Hence, by Corollary \ref{cor:6}, for $\chi_1=-M \wedge -1$,
	\begin{equation}\label{eq:83}
	p(t,x) \approx \big( t\vphi''(w) \big)^{\frac12} \exp \big\lbrace -t \big( w\vphi'(w)-\vphi(w) \big) \big\rbrace,
	\end{equation}
	if only $x\vphi^{-1}(1/t)<\chi_1$. In fact, if $\chi_1<-1$, then we also have
	\begin{equation}\label{eq:84}
	\big( t\vphi''(w) \big)^{\frac12} \exp \big\lbrace -t \big( w\vphi'(w)-\vphi(w) \big) \big\rbrace \approx \vphi^{-1}(1/t),
	\end{equation}
	for $\chi_1 \leq x\vphi^{-1}(1/t)\leq-1$. To show this, we first prove the following.
	\begin{claim}\label{clm:3}
		There exist $0<c_1 \leq 1 \leq c_2$ such that for all $t \in (0,c_1/\Phi(x_0))$ and $x \in (-\infty,x_1)$ satisfying
		\[
		\chi_1 \leq x\vphi^{-1}(1/t) \leq -1,
		\]
		we have
		\begin{equation*}
		-t\vphi' \big( \vphi^{-1}(c_2/t)\big) \leq x \leq -t\vphi' \big( \vphi^{-1}(c_1/t)\big).
		\end{equation*}
	\end{claim}
	Indeed, by Proposition \ref{prop:10}, there is $C_1 \geq 1$ such that for all $r>\Phi(x_0 \vee 2\theta_0)$,
	\begin{equation}\label{eq:60}
	C_1^{-1}\Phi^{-1}(r) \leq \vphi^{-1}(r) \leq C_1\Phi^{-1}(r).
	\end{equation}
	Let $c_2 = (-\chi_1C'C_1^2)^{\alpha/(\alpha-1)} \in [1,\infty)$, where $C'$ is taken from \eqref{eq:85}. Then it follows that
	\[
	c_2^{-1}\vphi^{-1}(c_2/t) \leq c_2^{(1-\alpha)/\alpha}C_1^2C' \vphi^{-1}(1/t) = (-\chi_1)^{-1}\vphi^{-1}(1/t).
	\]
	Thus, by Proposition \ref{prop:5},
	\[
	x \geq -\frac{-\chi_1}{\vphi^{-1}(1/t)} \geq - t\frac{\vphi \big( \vphi^{-1}(c_2/t) \big)}{\vphi^{-1}(c_2/t)} \geq  -t\vphi' \big( \vphi^{-1}(c_2/t) \big).
	\]
	Moreover, also by Proposition \ref{prop:5}, with $c_1=C^{-\alpha/(\alpha-1)}$ we have for $t \in (0,c_1/\vphi(x_0))$,
	\[
	t\vphi' \big( \vphi^{-1}(c_1/t) \big) \leq \frac{Cc_1}{\vphi^{-1}(1/t)} \cdot \frac{\vphi^{-1}(1/t)}{\vphi^{-1}(c_1/t)} \leq \frac{Cc_1^{(\alpha-1)/\alpha}}{\vphi^{-1}(1/t)},
	\]
	thus,
	\[
	x \leq -\frac{1}{\vphi^{-1}(1/t)} \leq -t\vphi' \big( \vphi^{-1}(c_1/t) \big),
	\]
	and the Claim follows.

	Now, using Claim \ref{clm:3}, Proposition \ref{prop:8}, \eqref{eq:17} and Proposition \ref{prop:10} we get that for all $t \in (0,c_1/\Phi(x_0))$,
	\begin{equation}\label{eq:86}
	w \approx \vphi^{-1}(1/t).
	\end{equation}
	Hence, in view of Proposition \ref{prop:5}, $tw\vphi'(w) \approx 1$ and consequently,
	\[
	\exp \big\lbrace -t\big( w\vphi'(w)-\vphi(w) \big) \big\rbrace \approx 1.
	\]
	Furthermore, by Proposition \ref{prop:3},
	\[
	\vphi''(w) \approx w\vphi'(w),
	\]
	which, combined with \eqref{eq:86}, yields
	\[
	\frac{1}{\sqrt{t\vphi''(w)}} \approx \frac{w}{\sqrt{w\vphi'(w)}} \approx \vphi^{-1}(1/t),
	\]
	and \eqref{eq:84} follows for $t \in (0,c_1/\Phi(x_0))$.
	
	Next, recall that $\theta_1=0$ and $\vphi'(\theta)=0$. Therefore, in view of \eqref{eq:30}, we in fact have
	\[
	b_r = -\int_r^{\infty} s\nu(s)\,ds.
	\] 
	Now, let $x>1/\vphi^{-1}(1/t)$. By Theorem \ref{thm:2}, Remark \ref{rem:4} and monotonicity of $\eta$,
	\[
	p (t,x) \lesssim t\eta\big(x-tb_{1/\Phi^{-1}(1/t)}\big) \leq t\eta(x).
	\]
	Thus, by Corollary \ref{cor:2}, for all $t \in (0,1/\Phi(x_0))$ and $x \in (0,x_1)$ such that $x\vphi^{-1}(1/t) > 1$,
	\begin{equation}\label{eq:81}
	p(t,x) \lesssim tx^{-1}\vphi(1/x).
	\end{equation}
	Next, by Lemma \ref{lem:4} and Proposition \ref{prop:11}, there are $M_0>0$, $\rho_0>0$ and $c>0$ such that for any $t \in (0,1/\Phi(x_0))$ and $x \in (0,x_1)$ satisfying $x\vphi^{-1}(1/t) \geq \rho_0$ we have
	\begin{equation}\label{eq:79}
	p(t,x) \geq ctx^{-1}\vphi(1/x).
	\end{equation}
	Thus, if we set $\chi_2 = 1 \vee \rho_0$, then by \eqref{eq:81} and \eqref{eq:79}, for all $t \in (0,1/\Phi(x_0))$ and $x \in (0,x_1)$ such that $x\vphi^{-1}(1/t) \geq \chi_2$,
	\begin{equation}\label{eq:82}
	p(t,x) \approx tx^{-1}\vphi(1/x).
	\end{equation}
	Finally, by Theorem \ref{thm:3}, for all $t \in (0,1/\Phi(x_0))$ and $x \in (-\infty,x_1)$ satisfying $\chi_1 < x\vphi^{-1}(1/t) < \chi_2$
	we have
	\begin{equation*}
	p(t,x) \approx \vphi^{-1}(1/t).
	\end{equation*}
	It remains to notice that if $\chi_2>1$, then, by scaling properties of $\vphi$, for all $t \in (0,1/\Phi(x_0))$ and $x \in (0,x_1)$ satisfying $1 \leq x\vphi^{-1}(1/t) \leq \chi_2$
	we have
	\[
	tx^{-1}\vphi(1/x) \approx \vphi^{-1}(1/t),
	\]
	which, combined with \eqref{eq:83}, \eqref{eq:84} and \eqref{eq:82} finishes the proof for the case $x_0=0$. If $x_0>0$ then we can use positivity and continuity to extend the time range from $c_1/\Phi(x_0)$ to $1/\Phi(x_0)$. 
\end{proof}
\begin{remark}\label{rem:5}
		Taking into account \eqref{eq:36}, Corollary \ref{cor:2} and Proposition \ref{prop:11} we may observe that in fact,
		\[
		\nu(x) \approx x^{-1}\vphi(1/x),
		\]
		for all $0<x<x_0^{-1} \wedge (2\theta_0)^{-1}$. Therefore, by inspecting the proof of Theorem \ref{thm:4} one can show that the term $tx^{-1}\vphi(1/x)$ in the thesis may be replaced by $t\nu(x)$.
\end{remark}
\begin{example}
	Let $\bfX$ be a spectrally positive $\alpha$-stable process with the Laplace exponent $\vphi(\lambda)=\lambda^{\alpha}$, where $\alpha > 1$. Then it is clear that we have  $\vphi''(\lambda)=\alpha(\alpha-1)\lambda^{\alpha-2}$ and $(\vphi')^{-1}(y) = \alpha^{-\alpha/(\alpha-1)}y^{-1/(\alpha-1)}$. Consequently, from Theorem \ref{thm:1} we get that the asymptotics of $p(t,x)$ is of the form
	\[
	\frac{1}{\sqrt{2\pi\alpha(1-\alpha)}} \bigg( \frac{-x}{\alpha} \bigg)^{-\alpha/2(\alpha-1)} t^{(2-\alpha)/2(\alpha-1)}\exp \Bigg\{ -(\alpha-1) t^{-1/(\alpha-1)} \bigg( \frac{-x}{\alpha}\bigg)^{-\alpha/(\alpha-1)} \Bigg\},
	\]
	which after setting $t=1$ coincides with \cite[Theorem 2.5.3]{Zolotarev64}. Moreover, by Theorem \ref{thm:4},
	\begin{align*}
	p(t,x) \approx \left\lbrace \begin{array}{lll}
	\frac{1}{\sqrt{2\pi\alpha(1-\alpha)}} \bigg( \frac{-x}{\alpha} \bigg)^{-\frac{\alpha}{2(\alpha-1)}} t^{\frac{2-\alpha}{2(\alpha-1)}}\exp \Bigg\{ -(\alpha-1) t^{-\frac{1}{\alpha-1}} \bigg( \frac{-x}{\alpha}\bigg)^{-\frac{\alpha}{\alpha-1}} \Bigg\}, & \text{if } & \frac{x}{t^{1/\alpha}} \leq -1,\\
	t^{-\frac{1}{\alpha}}, & \text{if } & -1 < \frac{x}{t^{1/\alpha}} \leq 1,\\ 
	\frac{t}{x^{1+\alpha}}, & \text{if } & \frac{x}{t^{1/\alpha}}>1.
	\end{array}	\right.
	\end{align*}
	For $\alpha=1$, in view of \cite[Proposition 1.2.12]{SamorodnitskyTaqqu94} we have $\vphi(\lambda)=\lambda \ln \lambda$. Therefore, $\vphi''(\lambda)=\lambda^{-1}$ and $(\vphi')^{-1}(y)=e^{y-1}$. By Theorem \ref{thm:1} we get the following form of the asymptotics:
	\[
	\frac{1}{\sqrt{2\pi t}} \exp \Bigg\{ -\frac{-x/t-1}{2}-e^{-x/t-1} \Bigg\},
	\]
	which again, after substituting $t=1$ coincides with \cite[Theorem 2.5.3]{Zolotarev64}. Unfortunately, Theorem \ref{thm:4} cannot be applied due to scaling condition with $\alpha=1$ only. Let us also note that for the case of Brownian motion, using Theorem \ref{thm:1} it is straightforward to obtain Gaussian density in the asymptotics.
\end{example}
\bibliographystyle{acm}
\bibliography{bib_file}
\end{document}